\documentclass[11pt]{amsart}
\usepackage{amsfonts, amsmath, enumerate, amssymb, xypic}

\newtheorem{theorem}{Theorem}[section]
\newtheorem{lemma}[theorem]{Lemma}
\newtheorem{corollary}[theorem]{Corollary}
\newtheorem{example}[theorem]{Example}

\newtheorem{proposition}[theorem]{Proposition}
\newtheorem{remark}[theorem]{Remark}

\newtheorem{definition}[theorem]{Definition}

\newcommand{\ncom}{\newcommand}
\ncom{\lrar}{\longrightarrow}
\ncom{\rar}{\rightarrow}
\ncom{\ov}{\overline}
\ncom{\m}{\mbox}
\ncom{\sta}{\stackrel}
\ncom{\comx}{{\mathbb C}}
\ncom{\A}{{\mathbb A}}
\ncom{\Cl}{{\mathbb C}}
\ncom{\Z}{{\mathbb Z}}
\ncom{\Q}{{\mathbb Q}}
\ncom{\R}{{\mathbb R}}
\ncom{\G}{{\mathbb G}}
\ncom{\al}{\alpha}
\ncom{\p}{{\mathbb P}}
\ncom{\E}{{\mathbb E}}
\ncom{\N}{{\mathbb N}}
\ncom{\K}{{\mathbb K}}
\ncom{\X}{{\mathbb X}}
\ncom{\f}{\frac}
\ncom{\cA}{{\mathcal A}}
\ncom{\cB}{{\mathcal B}}
\ncom{\cX}{{\mathcal X}}
\ncom{\cO}{{\mathcal O}}
\ncom{\cW}{{\mathcal W}}
\ncom{\cP}{{\mathcal P}}
\ncom{\cS}{{\mathcal S}}
\ncom{\cM}{{\mathcal M}}
\ncom{\cC}{{\mathcal C}}
\ncom{\cT}{{\mathcal T}}
\ncom{\cF}{{\mathcal F}}
\ncom{\cN}{{\mathcal N}}
\ncom{\cJ}{{\mathcal J}}
\ncom{\cK}{{\mathcal K}}
\ncom{\cV}{{\mathcal V}}
\ncom{\cZ}{{\mathcal Z}}
\ncom{\cU}{{\mathcal U}}
\ncom{\cSU}{{\mathcal S \mathcal U}}
\ncom{\cG}{{\mathcal G}}
\ncom{\cQ}{{\mathcal Q}}
\ncom{\cY}{{\mathcal Y}}
\ncom{\cE}{{\mathcal E}}
\ncom{\what}{\widehat}
\ncom{\delbar}{\overline{\partial}}

\ncom{\eop}{{\hfill $\Box$}}

\def \rmG{\rm G}
\def \rmX{\rm X}
\def \cX{\mathcal X}
\def \rmC{\rm C}
\def \cY{\mathcal Y}
\def \rmS{\rm S}
\def \rmH{\rm H}
\def \oK{\rm K}
\def \B{\mathcal B}
\def \BH{\rm BH}
\def \BB{\rm BB}

\def \eps{\epsilon}

\def \colimn{\underset {n \rightarrow \infty}  {\hbox {lim}}}
\def \colimm{\underset {m \rightarrow \infty}  {\hbox {lim}}}
\def \codim{\rm {codim}}
\def \cL{\mathcal L}

\def \H{\rm H}

\def \CH{\rm {CH}}

\def \Pic{\rm {Pic}}
\def \M{\rm M}
\def \bPic{\mathbf {Pic}}

\def \rmI{\rm I}
\def \rmp{\rm p}
\def \rmR{\rm R}
\def \NS{\rm {NS}}
\def \Q{\mathbb Q}

\def \ra{\rightarrow}
\def \Br{\rm {Br}}
\def \O{\mathcal O}

\def \res{respectively}
\def \cL{\mathcal L}
\def \SL{\rm SL}

\def \rmT{\rm T}

\def \rmX{\rm X}
\def \rmY{\rm Y}
\def \rmp{\rm p}
\def \rmU{\rm U}
\def \rmN{\rm N}
\def \rmE{\rm E}
\def \rmW{\rm W}
\def \rmS{\rm S}
\def \rmZ{\rm Z}
\def \EG{\rm EG}
\def \BG{\rm BG}
\def \BT{\rm BT}
\def \ET{\rm ET}
\def \rmB{\rm B}
\def \H{\rm H}
\def \rmM{\rm M}
\def \rmP{\rm P}
\def \rmG{\rm G}
\def \rmL{\rm L}
\def \rmQ{\rm Q}
\def \bd{\mathbf d}
\def \bZ{\mathbb Z}
\def \bF{\mathbf F}
\def \bN{\mathbb N}
\def \bG{\mathbb G}
\def \oSpec{\rm {Spec}}
\def \rmV{\rm V}
\newcommand{\Gr}{\operatorname{Gr}}

\newcommand{\GL}{\operatorname{GL}}
\newcommand{\bfd}{{\bf d}}

\newcommand{\bfdim}{{\bf dim}}
\newcommand{\bfi}{{\bf i}}
\newcommand{\bfI}{{\bf I}}
\newcommand{\Aut}{\operatorname{Aut}}
\newcommand{\End}{\operatorname{End}}
\newcommand{\id}{\operatorname{id}}

\def \Hom{{\mathcal H}om}

\def \p{\rm p}
\def \rmA{\rm A}
  
\begin{document}
\baselineskip=16pt

\title[Brauer groups of algebraic stacks and GIT-quotients]{Brauer groups of algebraic stacks and GIT-quotients}


\author[J. N. Iyer]{Jaya NN  Iyer}

\address{The Institute of Mathematical Sciences, CIT
Campus, Taramani, Chennai 600113, India}
\email{jniyer@imsc.res.in}

\author[R. Joshua]{Roy Joshua}
\address{Department of Mathematics, The Ohio State University, Columbus, Ohio, 43210, USA}
\email{joshua.1@math.osu.edu}

\begin{abstract}
In this paper we consider the Brauer groups of algebraic stacks and GIT quotients: the only algebraic stacks we consider
in this paper are quotient stacks $[\rmX/\rmG]$, where $\rmX$ is a smooth scheme of finite type over a field $k$, 
and $\rmG$ is a linear algebraic group over $k$ and acting on $\rmX$, as well as various moduli stacks of principal $\rmG$-bundles
on a smooth projective curve $\rmX$, associated to a reductive group $\rmG$. We also consider the Brauer groups of the corresponding coarse 
moduli spaces, which most often identify with the corresponding GIT-quotients. One conclusion that we seem to draw then
is that the Brauer groups (or their $\ell$-primary torsion parts, for a fixed prime $\ell$ different from $char (k)$)
of the corresponding stacks and coarse moduli spaces depend strongly on the Brauer group of the given scheme $\rmX$.

\end{abstract}
\maketitle

\setcounter{tocdepth}{1}
\tableofcontents


\section{Introduction and the Main Results}

The paper originated in an effort by the authors to study the Brauer groups of GIT quotients associated to actions of
reductive groups. While working on various examples,  we realized that it is preferable to adopt a more general framework and goal of studying also the
Brauer groups of various algebraic stacks that show up in this context.
\vskip .1cm

Though, for the most part, we work over a fixed separably closed field $k$ of arbitrary characteristic, there are indeed
some of our results that do not require this restriction and hold over any base field. Therefore, we will adopt the
following framework for considering the Brauer groups. We will start with a base field $k$ of arbitrary characteristic. Let $\ell$ denote a 
fixed prime different from $char (k)$ and let $\rmX$ denote a smooth  scheme of finite type over $k$: {\it we will always restrict to such schemes}. Then one begins with the {\it Kummer sequence}
\begin{equation}
 \label{Kummer.seq.0}
1 \ra \mu_{\ell^n}(1) \ra {\mathbb G}_m {\overset {\ell^n} \ra} {\mathbb G}_m \ra 1,
\end{equation}
which holds on the (small)  \'etale site ${\rmX}_{et}$ of ${\rmX}$, whenever $\ell$ is invertible in $k$. (See \cite[section 3]{Gr} or
\cite[p. 66]{Mi}.) Taking \'etale cohomology, we obtain corresponding long-exact sequence:
\begin{equation}
 \label{Kummer.seq.1}
\rightarrow \H^1_{et}({\rmX},{\mathbb G}_m) {\overset {\ell^n} \ra } \H^1_{et}({\rmX}, {\mathbb G}_m) \ra  \H^2_{et}({\rmX},\mu_{\ell^n}(1)) \rightarrow \H^2_{et}({\rmX}, {\mathbb G}_m) \rightarrow H^2_{et}({\rmX}, {\mathbb G}_m) \rightarrow \cdots, 
\end{equation}
which holds on the \'etale site when $\ell$ is invertible in $k$. 
\begin{definition}
 \label{coh.Br.grp.1}
The \textit{cohomological Brauer group} $\Br({\rmX})$ is the torsion subgroup of the cohomology group $\H^2_{et}({\rmX}, {\mathbb G}_m)$. In other words,
 $\Br({\rmX})\,=\,\H^2_{et}({\rmX}, {\mathbb G}_m)_{tors}$.
\end{definition}
By Hilbert's Theorem 90, and since $\rmX$ is assumed to be smooth, we obtain the isomorphisms: 
\begin{equation}
     \label{Hilb.90}
\Pic({\rmX}) \cong {\rm CH}^1(\rmX)  \cong \H^1_{et}({\rmX}, {\mathbb G}_m) \cong \H^{2, 1}_{\M}(X, {\mathbb Z}),
\end{equation}
where $\H^{2, 1}_{\M}(X, {\mathbb Z})$ denotes motivic cohomology (in degree $2$ and weight $1$) and ${\rm CH}^1$ denotes the
Chow group in codimension $1$.
Then one also obtains the short-exact sequence:
\begin{equation}
 \label{Kummer.seq.2}
0 \ra \Pic({\rmX})/\ell^n\cong \NS({\rmX})/\ell^n \ra \H^2_{et}({\rmX}, \mu_{\ell^n}(1)) \ra \Br({\rmX})_{\ell^n} \ra 0 ,
\end{equation}
where the map $\Pic({\rmX})/\ell^n = \H^{2,1}_{\M}(\rmX, {\mathbb Z}/\ell^n) \ra \H^2_{et}({\rmX}, \mu_{\ell^n}(1))$ is the
cycle map, and therefore, $\Br({\rmX})_{\ell^n}$ identifies with the cokernel of the cycle map. Thus it follows that
$\Br({\rmX})_{\ell^n}$ is trivial if and only if the above cycle map is surjective:
our approach to the Brauer group
adopted in this paper is to consider the above cycle map from motivic cohomology to etale cohomology, and involves a combination of motivic and
 \'etale cohomology techniques.

\vskip .2cm
Let $\rmG$ denote a not-necessarily connected linear algebraic group, defined over $k$, and acting on 
a quasi-projective scheme $\rmX$. 
Next we recall the framework of Borel-style equivariant \'etale cohomology, and Borel-style equivariant motivic cohomology.
For this we form
an ind-scheme $\{\EG ^{gm,m}{\underset {\rmG} \times}\rmX|m\}$ and then take its \'etale  cohomology, and also its
motivic cohomology when $\rmX$ is also assumed to be smooth.
One may consult \cite{Tot99}, \cite{MV99}, and also section 2 for more details. Here $\BG^{gm,m}$ is a finite dimensional approximation to
the classifying space of the linear algebraic group $\rmG$, and $\EG^{gm,m}$ denotes the universal principal $\rmG$-bundle over
$\BG^{gm,m}$. In the terminology of Definition ~\ref{defn:Adm-Gad}, $\EG^{gm,m} =\rmU_m$ and $\BG^{gm,m}=\rmU_m/\rmG$. We also assume that 
such a $\BG^{gm,m}$ exists, for every $m \ge 0$, as a quasi-projective scheme over the given base field. There are standard arguments
to prove that that the cohomology of the ind-schemes $\{\BG^{gm,m}|m \ge 0\}$, $\{\EG^{gm,m}\times_{\rmG}\rmX|m \ge 0\}$ are independent of the choice of the admissible gadgets $\{\rmU_m|m \ge 0\}$
that enter into their definition: see, for example, \cite[Appendix B]{CJ19}.
\vskip .1cm
{\it Let $\ell$ denote a fixed prime different from $char(k)$.}
Then we let $\rmH^{*, \bullet}_{\rmG, \M}(X, Z/\ell^n)$  denote the motivic 
cohomology of $\{\EG ^{gm,m}{\underset {\rmG} \times}\rmX|m\}$  defined as the homotopy inverse limit of the motivic
cohomology of the finite dimensional approximations $\EG ^{gm,m}{\underset {\rmG} \times}\rmX$, that is, defined
by the usual Milnor exact sequence relating $lim^1$ and $lim$ of the motivic   hypercohomology of the above finite 
dimensional approximations. (When $*=2i$ and $\bullet =i$, for a non-negative integer $i$, these identify with the usual (equivariant) Chow groups.) $\rmH^{*}_{\rmG, et}(X, \mu_{\ell^n}(\bullet))$
is defined similarly.

\vskip .2cm
Recall that for each fixed integer $i\ge 0$, one obtains the isomorphisms (for $m$ chosen, depending on $i$):
\[ \rmH^{2i,i}_{\rmG, \M}(\rmX, Z/\ell^n ) \cong \rmH^{2i,i}_{\M}(\EG^{gm,m}{\underset {\rmG} \times}\rmX, Z/\ell^n), m>>0 \mbox { and } \rmX \mbox { smooth}, \mbox{ and} \]
\[ \rmH^{2i,i}_{\rmG, et}(\rmX, \mu_{\ell^n} ) \cong \rmH^{2i}_{et}(\EG^{gm,m}{\underset {\rmG} \times}\rmX, \mu_{\ell^n}(i)), m>>0. \]
These show that one may define the {\it $\rmG$-equivariant Brauer group} of a $\rmG$-scheme $\rmX$ as follows:
\begin{definition}
 \label{equiv. Brauer.grp}
 $\Br_{\rmG}(\rmX) = \rmH^2_{et}(\EG^{gm,m}{\underset {\rmG} \times} \rmX, {\mathbb G}_m) _{tors}$, for $m>>0$, where the subscript $tors$ 
 denotes the torsion subgroup. \footnote{Here we remind the reader that, despite the similarity in appearance, the above equivariant Brauer
  groups are quite different from what are called, {\it invariant Brauer groups}: see \cite{Cao}.}
\end{definition}
Moreover, we obtain from the Kummer-sequence the short-exact sequence:
\begin{equation}
 \label{Kummer.seq.2}
0 \ra \Pic(\EG^{gm,m}{\underset {\rmG} \times}{\rmX})/\ell^n \ra \rmH^2_{et}(\EG^{gm,m}{\underset {\rmG} \times}{\rmX}, \mu_{\ell^n}(1)) \ra \Br(\EG^{gm,m}{\underset {\rmG} \times}{\rmX})_{\ell^n} = \Br_{\rmG}({\rmX})_{\ell^n}\ra 0 \,  \mbox{ and}
\end{equation} 
\vskip .1cm \noindent
where 
\vskip .1cm
$\Pic(\EG^{gm,m}{\underset {\rmG} \times}{\rmX})/\ell^n = coker (\Pic(\EG^{gm,m}{\underset {\rmG} \times}{\rmX}) {\overset {\ell^n} \ra} \Pic(\EG^{gm,m}{\underset {\rmG} \times}{\rmX})), \,$
\vskip .1cm
$\Br_{\rmG}({\rmX})_{\ell^n}$ = the $\ell^n$-torsion part of $\Br_{\rmG}({\rmX})$.
\vskip .1cm
The comparison theorem \cite[Theorem 1.6]{J20} shows that $\rmH^{*}_{\rmG, et}(X, \mu_{\ell^n}(\bullet))$ identifies with
$\rmH_{smt}^*([\rmX/\rmG], \mu_{\ell^n})$, which denotes the cohomology of the quotient stack $[\rmX/\rmG]$ computed on
the smooth site: see Proposition ~\ref{coh.comp} below for further details. This motivates the the following definition.
\begin{definition}
 \label{Br.grp.stacks}
 Given an Artin stack $S$ of finite type over $k$, we define its Brauer group to be $\rmH_{smt}^2(S, {\mathbb G}_m)_{tors}$,
  where $\rmH_{smt}^2(S, {\mathbb G}_m)$ denotes cohomology computed on the smooth site, and the subscript $tors$ denotes its torsion subgroup. We denote
 this by $\Br(S)$. For a fixed prime $\ell \ne char (k)$, we let $\Br(S)_{\ell^{n}}$ denote the $\ell^n$-torsion
  part of $\Br(S)$.
\end{definition}
Then our first result is the following, which shows the Brauer group of a quotient stack $[\rmX/\rmG]$, so defined, identifies with
the $\rmG$-equivariant Brauer group defined in Definition ~\ref{equiv. Brauer.grp}.
\begin{theorem}
\label{Br.grp.quot.stacks} Assume that 
$\rmX$ is a smooth scheme of finite type over the base field $k$, and  provided with an
  action by the linear algebraic group $\rmG$. Then, assuming the above terminology,
 \[\Br([\rmX/\rmG])_{\ell ^n} \cong \Br_{\rmG}({\rmX})_{\ell^n}.\]
 Therefore, $\Br_{\rmG}({\rmX})_{\ell^n}$ is intrinsic to the quotient stack $[\rmX/\rmG]$.
\end{theorem}

\vskip .2cm
{\it For the rest of the paper, we need to restrict to the case where the base field $k$ is separably closed.}
Our next main result is the following theorem and its corollary.
 \begin{theorem} 
 \label{nonequiv.to.equiv}
 Assume the base field $k$ is separably closed. Suppose $\rmX$ is a smooth scheme,  provided
 with an action by the {\it connected} linear algebraic group $\rmG$. Then the induced map 
 \[\rmH^{2,1}_{\rmG, \M}(\rmX, Z/\ell^n) \ra \rmH^2_{\rmG, et}(\rmX, \mu_{\ell^n}(1))\]
 is also an isomorphism  under any one of the following hypotheses:
 \begin{enumerate}[\rm(i)]
 \item The group $\rmG$ is a torus
  \item $k$ is perfect (which, in view of the assumption that it is separably closed, implies it is also algebraically closed). 
  The group $\rmG$ is {\it special} in the sense of Grothendieck (see \cite{Ch}). If $\rmW$  denotes the Weyl group associated to a maximal torus in $\rmG$,
  $|\rmW|$ (which is the order of $\rmW$) is relatively prime to $\ell$, and  the cycle map  induces an isomorphism
 \[\rmH^{2, 1}_{\M}(\rmX, Z/\ell^n) \ra \rmH^2_{et}(\rmX, \mu_{\ell^n}(1)),\]
  \item $\rmX = \rmG/\rmH$, for a closed connected linear algebraic subgroup $\rmH$ of $\rmG$, so that the torsion index of the group $\rmH$ is prime to $\ell$, where the torsion index of linear algebraic groups
  is discussed in \cite{Gr58}, \cite[section 1]{Tot05}.
 \end{enumerate}
\end{theorem}
 \begin{corollary}
 \label{cor.1} (i) Under the assumptions of Theorem ~\ref{nonequiv.to.equiv}(i) or (ii), if $\Br(\rmX)_{\ell^n} =0$, so is
   $\Br([\rmX/\rmG])_{\ell^n}$. More generally, if $\rmX$ is a smooth scheme of finite type over a perfect field $k$ and provided
   with the action by a split linear algebraic group $\rmG$ that is special, and if $\Br(\rmX^s)_{\ell^n}=0$, then so is  $\Br([\rmX^s/\rmG^s])_{\ell^n}$, provided $|\rmW|$ is prime to $\ell$.
   Here, for a scheme $\rmY$ over $k$, $\rmY^s$ denotes its base extension to the separable closure of $k$ and $\rmW$ denotes the Weyl group associated to a split maximal torus in $\rmG$.
 \vskip .1cm \noindent
 (ii) Assume the base field $k$ is separably closed. If $\rmH$ is a connected linear algebraic group whose torsion index is prime to $\ell$, then $\Br(\BH)_{\ell^n}=0$, where
 $\BH$ denotes the classifying stack of $\rmH$, that is $[Spec \, k/\rmH]$.
 \end{corollary}
 \begin{remark} Observe that if $\rmX$ is a projective smooth variety that is rational, then $\Br(\rmX)_{\ell^n} =0$.
  The Corollary then shows that, under the assumptions of Theorem ~\ref{nonequiv.to.equiv}(i),  $\Br([\rmX/\rmG])_{\ell^n}=0$ as well. We will show in section 4  that Corollary ~\ref{cor.1} provides a very quick proof of the triviality of the 
  $\ell^n$-torsion part of the Brauer group of the moduli stack of elliptic curves, for any prime $\ell$ different from $char(k)$, as long as $k$ is separably closed, and $char(k) \ne 2, 3$: see also see Theorem ~\ref{mod.ell.curves}.
 \end{remark}
\vskip .1cm
In case $\rmG$ is {\it not} connected, one has the following extension of Theorem ~\ref{nonequiv.to.equiv}. Let $\rmG \ra \tilde \rmG$
denote an imbedding of $\rmG$ as a closed subgroup of a connected linear algebraic group. In particular, it applies to the case when
 $\rmG$ is a finite group.
 \begin{theorem} 
 \label{nonequiv.to.equiv.1}
 Assume the base field $k$ is algebraically closed. 
 Suppose $\rmX$ is a smooth scheme,  provided
 with an action by the {\it not-necessarily connected} linear algebraic group $\rmG$. Then the induced map 
 \[\rmH^{2,1}_{\rmG, \M}(\rmX, Z/\ell^n) \ra \rmH^2_{\rmG, et}(\rmX, \mu_{\ell^n}(1))\]
 is also an isomorphism  under the following hypotheses:
\vskip .1cm
 $|\tilde \rmW|$  is relatively prime to $\ell$ and  the cycle map  induces an isomorphism
 \[\rmH^{2, 1}_{\M}(\tilde \rmG\times_{\rmG}\rmX, Z/\ell^n) \ra \rmH^2_{et}(\tilde \rmG\times_{\rmG}\rmX, \mu_{\ell^n}(1)),\]
 where $\rmG$ imbeds as a closed subgroup of the connected linear algebraic group $\tilde \rmG$ which is also assumed to be {\it special}, $\tilde \rmW$ is the Weyl group
 of $\tilde \rmG$, and $|\tilde \rmW|$ is the order of $\tilde \rmW$.
\end{theorem}

\vskip .1cm
Section 4 is devoted to considering various examples making use of the above theorems. We will summarize here
a few of these, and one may consult section 4 for additional examples and details.
\vskip .1cm
Next let $\rmX$ denote a smooth projective curve of genus $g$ over $k$, provided with a $k$-rational point. 
Then one knows the isomorphism of stacks (see for example, \cite[Proposition 4.2.5]{Wang}):
\begin{equation}
 \label{Bun1}
 {\rm Bun}_{1, d}(\rmX) \cong {\rm B}{\mathbb G}_m^{gm} \times {\bold Pic}^d(\rmX),
\end{equation}
where ${\rm B}{\mathbb G}_m^{gm} = \colimn {\rm B}{\mathbb G}^{gm,n}_m$, 
${\rm Bun}_{1, d}(\rmX)$ denotes the moduli stack of line bundles of degree $d$ on $\rmX$ and $\bPic^d(\rmX)$ denotes
the Picard scheme. 
In view of the above isomorphism of stacks, one may define the Brauer group of the stack ${\rm Bun}_{1, d}(\rmX)$ to be
the Brauer group of the stack ${\rm B}{\mathbb G}_m^{\rm gm} \times {\bPic}^d(\rmX)$.
Then, we obtain the following theorem.
\begin{theorem} Assume the base field $k$ is separably closed. Then, assuming the above situation, we obtain the isomorphism:
 \label{main.thm.2}
 \[\Br({\rm Bun}_{1, d}(\rmX))_{\ell ^n} \cong \Br({\bold Pic}^d(\rmX))_{\ell ^n} \cong \Br({\rm Sym}^d(\rmX))_{\ell ^n}.\]
In particular, $\Br({\rm Bun}_{1, d}(\rmX))_{\ell ^n} \cong 0$ if $\rmX$ is rational.
\end{theorem}
\vskip .1cm
Next we consider the moduli stack of elliptic curves, which has a nice presentation as a quotient stack for the action of ${\mathbb G}_m$: see \cite{Ols}.
\begin{theorem}
\label{mod.ell.curves}
Let ${\mathcal M}_{1,1}$ denote the moduli stack of elliptic curves over the base field $k$ (which we recall is separably closed). Assume that
  $char(k) \ne 2,3$ and $\ell$ is a prime different from $char(k)$. Then $\Br({\mathcal M}_{1,1})_{\ell^n}=0$.
\end{theorem}
\vskip .1cm \noindent
This theorem is essentially Theorem ~\ref{ell.curves} and a quick proof of that theorem is discussed in section 4.
\vskip .2cm
Next we shift our focus to the Brauer groups of GIT quotients for actions of connected reductive groups $\rmG$ on {\it smooth} schemes, with the
 base field assumed to be {\it algebraically closed} of arbitrary characteristic $p \ge 0$. 
In this context, 
one recalls that such schemes admit a $\rmG$-stable stratification (that is, a decomposition into locally closed smooth and
$\rmG$-stable subschemes) based on stability considerations: see \cite{Kir84} and \cite{ADK}.
Of key importance to us is that the stratification of the above scheme $\rmX$ provides the long exact sequences. We consider
this is in a slightly more general context as follows. Let $\rmU \subseteq \rmX$ denote an open $\rmG$-stable subscheme. Then one
obtains the long exact (localization) sequences
\vskip .1cm 
 \begin{equation}
\label{mot.exact.seq}
\cdots \ra \H_{\rmG, \rmX- \rmU, \M}^{2i, i}(\rmX, Z/\ell^n) \ra \H_{\rmG, \M}^{2i, i}(X, Z/\ell^n) \ra 
 \rmH^{2i, i}_{\rmG, \M}( \rmU, Z/\ell^n) \ra  \H_{\rmG, \rmX- \rmU, \M}^{2i+1, i}(\rmX, Z/\ell^n) \ra \cdots,
 \end{equation} 
 \vskip .2cm
\begin{equation}
 \label{et.exact.seq}
\cdots \ra \H_{\rmG, \rmX- \rmU, et}^{2i}(\rmX, \mu_{\ell^n}(i)) \ra \H_{\rmG, et}^{2i}(X, \mu_{\ell^n}(i)) \ra 
 \rmH^{2i}_{\rmG, et}( \rmU, \mu_{\ell^n}(i)) \ra \H_{\rmG, \rmX- \rmU, et}^{2i+1}(\rmX, \mu_{\ell^n}(i)) \ra \cdots.
 \end{equation}
We are particularly interested in the situation when the above long exact sequences break up into short exact sequences at $i=1$, 
that is, where the maps 
 \begin{align}
\label{surj.ss}
\H_{\rmG, \M}^{2, 1}(X, Z/\ell^n) &\ra \rmH^{2, 1}_{\rmG, \M}( \rmU, Z/\ell^n)\\
\H_{\rmG, et}^{2}(X, \mu_{\ell^n}(1)) &\ra \rmH^{2}_{\rmG, et}( \rmU, \mu_{\ell^n}(1)) \notag
\end{align}
in both the above long exact sequences are surjections.
This is not always the case with finite coefficients, so our first result is to say, when in fact, 
one obtains the above results for {\it suitable choice of finite 
coefficients}. 
\vskip .1cm
Let $\{\rmS_{\beta}|\beta\}$ denote the stratification of the given smooth scheme $\rmX$
 defined in \cite{Kir84} based on stability considerations in the context of geometric invariant theory. We will then adopt the following terminology from \cite{Kir84}: 
for each $\beta \eps \B$, let $\rmY_{\beta}$ denote a locally closed subscheme of $\rmS_{\beta}$ so that it is stabilized
by a parabolic subgroup $\rmP_{\beta}$, with Levi factor $\rmL_{\beta}$. Moreover, then 
$\rmS_{\beta} \cong \rmG{\underset {\rmP_{\beta}} \times}Y_{\beta}^{ss}$, and there is a scheme $\rmZ_{\beta}$ with
an $\rmL_{\beta}$-action and an $\rmL_{\beta}$-equivariant Zariski-locally trivial surjection $\rmY_{\beta}^{ss} \ra \rmZ_{\beta}^{ss}$ whose 
fibers are affine spaces. Moreover, $\rmZ _{\beta}$ is a smooth locally closed $\rmL_{\beta}$ -stable subscheme of $\rmX$, so that it is
a component of the fixed point scheme for the induced action by a $1$-parameter subgroup $\rmT'_{\beta}$ of $\rmL_{\beta}$. 
Finally, it is important to observe that the normal bundle to the stratum $\rmS_{\beta}$ in $\rmX$ is a quotient of the restriction of 
the normal bundle 
to $\rmZ_{\beta}$ in $\rmX$. If $\rmS_{\beta}$ denotes a stratum on $\rmX$ for the stratification considered above (based on stability), $\rmW_{\beta}$
 will denote the Weyl group corresponding to the Levi subgroup $\rmL_{\beta}$.
\vskip .1cm
We will also let $\rmS_{\beta_o}$ denote any one of the strata in $\rmX- \rmX^{ss}$ which are of the highest dimension (and hence
 open in $\rmX - \rmX^{ss}$). We will denote the corresponding Weyl group in $\rmL_{\beta_o}$ by $\rmW_{\beta_o}$.
 \vskip .1cm
Then the following are some of hypotheses we may impose on the given scheme $\rmX$ and the given action by the linear algebraic group
 $\rmG$.
 \begin{enumerate}[\rm(i)]
  \item $\rmX$ is $\rmG$-projective with a {\it manageable} $\rmG$-linearized action on $\rmX$ in the sense of \cite[Theorem 4.7]{ADK}. (Observe that 
 the condition on the action being manageable is automatically satisfied if $char(k) =0$).
  \item Alternatively, $\rmX$ is the affine space of representations of a fixed quiver $\rmQ$ with dimension vector $\bd$.
\end{enumerate}

\begin{theorem}
\label{main.thm.3}
Assume the base field $k$ is algebraically closed, the reductive group $\rmG$ is {\it connected}, $\rmX$ is a smooth $\rmG$-scheme, and  $\ell$ is a prime
different from $char (k)$. Assume further that one of the following hypothesis also holds:
\vskip .1cm
(a) $\rmU = \rmX^{ss}$ and $codim_{\rmX}(\rmX -\rmX^{ss}) \ge 2$ {\it or}
\vskip .1cm
(b) we are in one of the two cases considered in (i) or (ii) above, 
$\rmU = \rmX^{ss}$ and if $\rmS_{\beta_o}$ denotes any stratum in $\rmX - \rmX^{ss}$ of the highest dimension in $\rmX- \rmX^{ss}$, then 
$\ell$ is relatively prime to $|\rmW_{\beta_o}|$, which denotes the order of the Weyl group $\rmW_{\beta_o}$.
\vskip .1cm \noindent
Then the maps in ~\eqref{surj.ss} are surjections.
 \end{theorem}

We will let $\rmW$ denote the Weyl group in $\rmG$ and $|\rmW|$ will denotes its order.

 \begin{theorem}
 \label{cor.2}
 Assume that the cycle map, $cycl: \rmH^{2, 1}_{ \M}(\rmX, Z/\ell^n) \ra \rmH^{2}_{et}(\rmX, \mu_{\ell^n}(1))$ is an
 isomorphism (or equivalently, $Br(\rmX)_{\ell^n} =0$) and that the group $\rmG$ is special.
 \vskip .1cm 
  Then the cycle map 
 $cycl: \rmH^{2,1}_{\rmG, \M}(\rmX^{ss}, Z/\ell^n) \ra \rmH^{2}_{\rmG, et}(\rmX^{ss}, \mu_{\ell^n}(1))$ is a surjection
 (or equivalently,  \\${\rm Br}_{\rmG}(\rmX^{ss})_{\ell^n}=0$), provided
 $|\rmW|$ is relatively prime to $\ell$, 
 and one of the following additional hypotheses holds:
 \vskip .1cm 
 (i) the assumptions as in case (a) of Theorem ~\ref{main.thm.3} holds  { or }
 \vskip .1cm
 (ii) the assumptions as in case (b) of Theorem ~\ref{main.thm.3} holds.
\end{theorem}

 \begin{theorem}
 \label{Br.triv.2}
 Assume in addition to the hypotheses of the last theorem that one of the following holds:
 \begin{enumerate}
 \item $\rmX^s = \rmX^{ss}$, that is,  the subscheme of semi-stable points on $\rmX$ is equal to the subscheme  of
 stable points on $\rmX$, \it {or }
 \item $codim_{\rmX}(\rmX^{ss}-\rmX^s) \ge 2$. 
\end{enumerate}
 Then, ${\rm Br}(\rmX//\rmG)_{\ell^n} =0$, if $\ell$ is also prime to the orders of the stabilizers at points on $\rmX^s$, where
 $\rmX//\rmG = \rmX^s/\rmG$ denotes the GIT quotient of $\rmX$ by $\rmG$.
\end{theorem}
\vskip .1cm \noindent
Section 6 discusses various examples where the above theorems are utilized.
\vskip .1cm
{\bf Acknowledgment}. The second author thanks the Institute for Mathematical Sciences, Chennai, for supporting his visit during the
summer of 2019 and for providing a pleasant working environment during his visit. He also thanks Ajneet Dhillon for 
helpful discussions that have contributed to the paper.
\section{Equivariant Brauer groups vs. Brauer groups of quotient stacks: Proof of Theorem ~\ref{Br.grp.quot.stacks}}
The goal of this section is to prove Theorem ~\ref{Br.grp.quot.stacks}. We begin discussing the construction of geometric
classifying spaces and Borel construction followed by the simplicial variant. Throughout this section, $k$ will denote
any field.
\vskip .2cm
\subsection{Admissible gadgets}
\label{subsubsection:adm.gadgets}
Let $\rmG$ denote a fixed linear algebraic group over $k$. 
 We will define a pair $(\rmW,{\rm U})$ of smooth varieties over $k$
to be  a {\sl good pair} for $\rmG$ if $\rmW$ is a $k$-rational representation of $\rmG$ and
${\rm U} \subsetneq \rmW$ is a $\rmG$-invariant non-empty open subset on which $\rmG$ acts freely and so that ${\rm U}/\rmG$ is a variety. 
It is known ({\sl cf.} \cite[Remark~1.4]{Tot99}) that a 
good pair for $\rmG$ always exists.
\begin{definition}
\label{defn:Adm-Gad}
A sequence of pairs $ \{ (\rmW_m, {\rm U}_m)|m \ge 1\}$ of smooth varieties
over $k$ is called an {\sl admissible gadget} for $\rmG$, if there exists a
good pair $(\rmW,{\rm U})$ for $\rmG$ such that $\rmW_m = \rmW^{\times^ m}$ and ${\rm U}_m \subsetneq \rmW_m$
is a $\rmG$-invariant open subset such that the following hold for each $m \ge 1$.
\begin{enumerate}
\item
$\left({\rm U}_m \times \rmW\right) \cup \left(\rmW \times {\rm U}_m\right)
\subseteq {\rm U}_{m+1}$ as $\rmG$-invariant open subvarieties.
\item
$\{\codim _{{\rm U}_{m+1}}\left({\rm U}_{m+1} \setminus \left({\rm U}_{m} \times \rmW\right)\right)|m\}$ is 
a strictly increasing sequence,  
\newline \noindent
that is, \[\codim_{{\rm U}_{m+2}}\left({\rm U}_{m+2} \setminus 
\left({\rm U}_{m+1} \times W\right)\right) > 
\codim_{{\rm U}_{m+1}}\left({\rm U}_{m+1} \setminus \left({\rm U}_{m} \times W\right)\right).\]
\item
$\{\codim_{\rmW_m}\left(\rmW_m \setminus {\rm U}_m\right)|m\}$ is a strictly increasing sequence, that is, 
\[\codim_{\rmW_{m+1}}\left(\rmW_{m+1} \setminus {\rm U}_{m+1}\right)
> \codim_{\rmW_m}\left(\rmW_m \setminus {\rm U}_m\right).\]
\item
${\rm U}_m$ has a free $\rmG$-action, the quotient ${\rm U}_m/\rmG$ is  a smooth quasi-projective
variety over $k$ and ${\rm U}_m \ra {\rm U}_m/\rmG$ is a principal $\rmG$-bundle.
\end{enumerate}
\end{definition}
\begin{lemma}
\label{smooth.quotients.1} Let ${\rm U}$ denote a smooth quasi-projective scheme over a field $\oK$ with a free action by the linear algebraic group 
$\rmG$ so that the quotient ${\rm U}/\rmG$ exists as a smooth quasi-projective scheme over $\oK$. Then if $\rmX$ is any smooth $\rmG$-quasi-projective  scheme
over $\oK$, the quotient ${\rm U}{\underset {\rmG} \times}\rmX \cong ({\rm U} {\underset {\oSpec \, \oK} \times}\rmX)/\rmG$ (for the diagonal action of $\rmG$) exists  as a  scheme
over $\oK$. 
\end{lemma}
\begin{proof} 
This follows, for
example, from \cite[Proposition 7.1]{MFK94}.
\end{proof}
\vskip .2cm
 An {\it example} of an admissible gadget for $\rmG$ can be constructed as follows: start  with a good pair $(\rmW, {\rm U})$ for $\rmG$. 
The choice of such a good pair will vary depending on $\rmG$. 
 Choose a faithful $k$-rational representation ${\rm R}$ of $\rmG$ of dimension $n$, that is, $\rmG$ admits a closed
immersion into $\GL(R)$. Then 
$\rmG$ acts freely on an open subset ${\rm U}$ of ${\rm W}= {\rm R}^{\oplus n} = {\rm {End}}({\rm R})$ so that ${\rm U}/\rmG$ is a variety. (For e.g.  ${\rm U}=\GL(R)$.)
Let $\rmZ = \rmW \setminus {\rm U}$.
\vskip .2cm
Given a good pair $(\rmW, {\rm U})$, we now let
\begin{equation}
\label{adm.gadget.1}
 \rmW_m = \rmW^{\times m}, {\rm U}_1 = {\rm U} \mbox{ and } {\rm U}_{m+1} = \left({\rm U}_m \times \rmW \right) \cup
\left(\rmW \times {\rm U}_m\right) \mbox{ for }m \ge 1.
\end{equation} 
Setting 
$\rmZ_1 = \rmZ$ and $\rmZ_{m+1} = {\rm U}_{m+1} \setminus \left({\rm U}_m \times \rmW\right)$ for 
$m \ge 1$, one checks that $\rmW_m \setminus {\rm U}_m = \rmZ^m$ and
$\rmZ_{m+1} = \rmZ^m \times {\rm U}$.
In particular, $\codim _{\rmW_m}\left(\rmW_m \setminus {\rm U}_m\right) =
m (\codim_{\rmW}(\rmZ))$ and
$\codim _{{\rm U}_{m+1}}\left(\rmZ_{m+1}\right) = (m+1)d - m(\dim(\rmZ))- d =  m (\codim _{\rmW}(\rmZ))$,
where $d = \dim (\rmW)$. Moreover, ${\rm U}_m \to {{\rm U}_m}/\rmG$ is a principal $\rmG$-bundle and the quotient $\rmV_m={\rm U}_m/\rmG$ exists 
 as a smooth quasi-projective scheme. 
\subsection{The geometric and simplicial Borel constructions}
\label{Borel.constr}
Given an admissible gadget $\{(\rmW_m, \rmU_m)|m \ge 0\}$ for the linear algebraic group $\rmG$ and a  $\rmG$-scheme $\rmX$,
we define
\begin{align}
 \label{geom.Borel}
 \EG^{gm,m} =\rmU_m, \quad \EG^{gm,m}\times_{\rmG}\rmX &= \rmU_m\times_{\rmG}\rmX, \quad \BG^{gm,m} = \rmU_m\times_{\rmG} (Spec \, k), \mbox{ and } \\
 \pi_m &:\EG^{gm,m}\times_{\rmG}\rmX  \ra \BG^{gm,m}. \notag
\end{align}
\vskip .1cm \noindent
The ind-scheme $\{\EG^{gm,m}\times_{\rmG}\rmX|m\ge 0\}$ is the {\it geometric Borel construction}. 
We will often denote $\colimm \{\EG^{gm,m} \times_{\rmG} \rm X|m \ge 0\}$ by $\EG^{gm}\times_{\rmG}\rmX$.
We next consider $\EG \times_{\rmG} \rmX$ which
is the simplicial scheme defined by $\rmG^n \times \rmX$ in degree $n$, and with the structure maps defined as follows:
\begin{align}
 \label{simp.Borel}
 d_i(g_0, \cdots, g_n, x) &= (g_1, \cdots, g_n, x), i=0\\
                          &= (g_1, \cdots, g_{i-1}. g_i, \cdots, g_n, x), 0<i<n \notag \\
                          &= (g_1, \cdots, g_{n-1}, g_n.x), i=n, \mbox{ and } \notag\\
s_i(g_0, \cdots, g_{n-1}, x) &= (g_0, \cdots, g_{i-1}, e, g_i, \cdots, x) \notag
\end{align}
where $g_i \in \rmG$, $x \in \rmX$, $g_{i-1}.g_i$ denotes the product of $g_{i-1}$ and $g_i$ in $\rmG$, while 
$g_n.x$ denotes the product of $g_n$ and $x$. $e$ denotes the unit element in $\rmG$. This is the {\it simplicial Borel construction}.
Then we obtain the following identification, which is well-known.
\begin{lemma}
 \label{simpl.Borel}
 One obtains an isomorphism: $\EG\times_{\rmG} \rmX \cong cosk_0^{[\rmX/\rmG]}(\rmX)$, where $cosk_0^{[\rmX/\rmG]}(\rmX)$
 is the simplicial scheme defined in degree $n$ by the $(n+1)$-fold fibered product of $\rmX$ with itself over the stack $[\rmX/\rmG]$,
 with the structure maps of the simplicial scheme $cosk_0^{[\rmX/\rmG]}(\rmX)$ induced by the above fibered products.
 \end{lemma}
 \vskip .1cm
For each fixed $m \ge 0 $, we obtain  the diagram of simplicial schemes (where $p_1$ is induced by the
projection $\EG^{gm, {\it m}} \times \rmX \ra \rmX$ and $p_2$ is induced by the projection 
$\EG \times (\EG^{gm, {\it m}} \times \rmX) \ra \EG^{gm, {\it m}} \times \rmX$):
\begin{equation}
  \xymatrix{&{\EG \times_{{\rmG}} (\EG^{gm, {\it m}} \times \rmX)} \ar@<1ex>[ld]_{\rm p_1} \ar@<-1ex>[rd]^{\rm p_2}\\
{\EG \times_{{\rmG}} \rmX}   && {\EG^{gm, {\it m}} {\underset {G} \times}\rmX}}
     \end{equation} 
\vskip .1cm \noindent
$\rmG$ acts diagonally on $\EG \times_{{\rmG}} (\EG^{gm, {\it m}} \times \rmX)$.
\begin{proposition}
\label{comp.isoms}
 (i) The map
\begin{align}
 \rmp_1^*: \rmH^1_{et}(\EG\times_{\rmG}\rmX, {\mathbb G}_m) &\ra \rmH^1_{et}(\EG \times_{\rmG} (\EG^{gm, {\it m}} \times \rmX), {\mathbb G}_m) \mbox{ and the map }\\
 \rmp_2^*: \rmH^1_{et}(\EG^{gm,m}\times_{\rmG}\rmX, {\mathbb G}_m) &\ra \rmH^1_{et}(\EG \times_{\rmG} (\EG^{gm, {\it m}} \times \rmX), {\mathbb G}_m), \mbox{ for m sufficiently large,}\notag
\end{align}
are isomorphisms. 
\vskip .1cm
(ii) The corresponding maps, \mbox{ for m sufficiently large,} with $\ell \ne char (k)$,
\begin{align}
 \rmp_1^*: \rmH^2_{et}(\EG\times_{\rmG}\rmX, \mu_{\ell^n}(1)) &\ra \rmH^2_{et}(\EG \times_{\rmG} (\EG^{gm, {\it m}} \times \rmX), \mu_{\ell^n}(1)), \mbox{ and }\\
 \rmp_2^*: \rmH^2_{et}(\EG^{gm,m}\times_{\rmG}\rmX, \mu_{\ell^n}(1)) &\ra \rmH^2_{et}(\EG \times_{\rmG} (\EG^{gm, {\it m}} \times \rmX), \mu_{\ell^n}(1)) \notag
\end{align}
are isomorphisms.
\end{proposition}
\begin{proof} The isomorphisms in (i) are rather involved, and therefore, we discuss the proof of (i) first.
A key to the proof is the observation that, over a base field $k$ which is separably closed, $\rmH^1_{et}({\mathbb A}^n, {\mathbb G}_m) \cong 0$, for any $n \ge 0$.
We consider the Leray spectral sequences associated to the maps $p_1$ and $p_2$:
\begin{align}
 \rmE_2^{s,t}(1) = \rmH^s_{et}(\EG\times_{\rmG} \times X, R^t \rmp_{1*}({\mathbb G}_m)) &\Longrightarrow \rmH^{s+t}_{et}(\EG \times_{\rmG} (\EG^{gm, {\it m}} \times \rmX), {\mathbb G}_m) \mbox{ and }\\
 \rmE_2^{s,t}(2) = \rmH^s_{et}(\EG^{gm,m}\times_{\rmG} \times X, R^t \rmp_{2*}({\mathbb G}_m)) &\Longrightarrow \rmH^{s+t}_{et}(\EG \times_{\rmG} (\EG^{gm, {\it m}} \times \rmX), {\mathbb G}_m).\notag
\end{align}
\vskip .1cm
Since $s, t \ge 0$, both spectral sequences converge strongly. 
\vskip .1cm
The stalks of 
$\rmR^t \rmp_{2*}({\mathbb G}_m) \cong \rmH^t(\EG{\underset {Spec \, k} \times} (Spec \, \rmA), {\mathbb G}_m)$, where $\rmA$ denotes
a strict Hensel ring. (Strictly speaking, in order to obtain the above identification, 
we need to make use of the simplicial topology as in \cite{J02} or \cite[5.4]{J20}. But we will ignore this rather subtle point for
the rest of the discussion.) Since $\EG \cong cosk_0^{Spec \, k}(\rmG)$, $\EG{\underset {Spec \, k} \times} (Spec \, \rmA) \cong cosk_0^{Spec \, \rmA}(\rmG\times_{Spec \, k}Spec \, A)$
is a {\it smooth hypercover} of ${\rm Spec} \, \rmA$. Therefore, we obtain the isomorphism:
\begin{equation}
\label{stalks.p2.1}
\rmH^t_{et}(\EG{\underset {Spec \, k} \times} (Spec \, A), {\mathbb G}_m) \cong \rmH^t_{smt}(\EG{\underset {Spec \, k} \times} (Spec \, A), {\mathbb G}_m) \cong \rmH^t(Spec \, A, {\mathbb G}_m).
\end{equation}
These groups are trivial for  $t=1$ (see, for example, \cite[Lemma 4.10]{Mi}). Therefore, it follows that for $t=1$, ${\rm R}^t \rmp_{2*}({\mathbb G}_m)_{Spec \, A} \cong 0$.
\vskip .2cm
Next we observe the isomorphism, by taking $t=0$ in ~\eqref{stalks.p2.1}:
\begin{equation}
 \label{stalks.p2.0}
 \rmp_{2*}({\mathbb G}_m)_{Spec \, A} \cong \rmH^0(\EG{\underset {Spec \, k} \times} (Spec \, A), {\mathbb G}_m) \cong \rmH^0(Spec \, A, {\mathbb G}_m),
\end{equation}
\vskip .1cm \noindent
where $\rmp_{2*}({\mathbb G}_m)_{Spec \, A}$ denotes the stalk of the sheaf $\rmp_{2*}({\mathbb G}_m)$ at ${\rm Spec }\, \rm A$.
Observing that ${\mathbb G}_m$ is in fact a sheaf on the flat site, and therefore also on the smooth site,
it follows that there is a natural map of sheaves ${\mathbb G}_m \ra \rmp_{2*}({\mathbb G}_m)$, where the ${\mathbb G}_m$ on the 
left (on the right) denotes the sheaf ${\mathbb G}_m$ restricted to the \'etale site of $\EG^{gm,m}\times_{\rmG}\rmX$ (the
\'etale site of ${\EG\times_{{\rmG}} (\EG^{gm, {\it m}} \times \rmX)}$, \res). 
The isomorphism in ~\eqref{stalks.p2.0} shows this
map induces an isomorphism stalk-wise. It follows that 
the natural map ${\mathbb G}_m \ra \rmp_{2*} ({\mathbb G}_m)$ of sheaves on the \'etale site is an isomorphism.
This provides the isomorphism:
\begin{equation}
\label{E2.p2}
\rmE_2^{1,0} = \rmH^1_{et}(\EG^{gm,m}\times_{\rmG}\rmX, \rmp_{2*}({\mathbb G}_m)) \cong \rmH^1_{et}(\EG^{gm,m}\times_{\rmG}\rmX, {\mathbb G}_m), m>>0.
\end{equation}
\vskip .2cm
The stalks of 
$\rmR^t \rmp_{1*}({\mathbb G}_m) \cong \rmH^t(\EG^{gm,m}{\underset {Spec \, k} \times} (Spec \, \rmA), {\mathbb G}_m)$, where $\rmA$ denotes
a strict Hensel ring, for all $t\ge 0$. Observe that this strict Hensel ring $\rmA$ is the stalk of the structure sheaf 
of $(\EG\times_{\rmG} \rmX)_n =\rmG^n \times \rmX$, at a geometric point. Hence it is a filtered direct limit $\lim_i \rmA_i$,
with each $\rmA_i$ regular. 
\vskip .1cm
To determine the groups $\rmH^t(\EG^{gm,m}{\underset {Spec \, k} \times} (Spec \, A), {\mathbb G}_m)$, we consider the long exact sequence (with $\EG^{gm,m} = \rmU_m$, which is assumed to be an open subscheme of the
 affine space ${\mathbb A}^m$, with $\rmZ_m = {\mathbb A}^m- \rmU_m$):
\begin{multline}
 \label{local.seq}
   \begin{split}
\cdots \ra \rmH^0_{\rmZ_m\times_{Spec \, k}Spec \, A,et}({\mathbb A}^m \times_{Spec \, k}Spec \, A, {\mathbb G}_m) \ra \rmH^0_{et}({\mathbb A}^m \times_{Spec \, k}Spec \, A, {\mathbb G}_m) \ra \\
{\overset {\alpha } \ra} \rmH^0_{et}(\rmU_m \times_{Spec \, k}Spec \, A, {\mathbb G}_m) \ra \rmH^1_{\rmZ_m\times_{Spec \, k}Spec \, A,et}({\mathbb A}^m \times_{Spec \, k}Spec \, A, {\mathbb G}_m)\\
\ra \rmH^1_{et}({\mathbb A}^m \times_{Spec \, k}Spec \, A, {\mathbb G}_m) {\overset {\beta} \ra} \rmH^1_{et}(\rmU_m \times_{Spec \, k}Spec \, A, {\mathbb G}_m) \ra \\
\ra \rmH^2_{\rmZ_m\times_{Spec \, k}Spec \, A,et}({\mathbb A}^m \times_{Spec \, k}Spec \, A, {\mathbb G}_m) \ra \cdots
   \end{split}
\end{multline}
\vskip .1cm
Next we observe the following isomorphisms for any smooth or regular quasi-projective scheme $\rmY$:
\begin{align}
 \label{CH1.isoms}
 \rmH^1_{et}(\rmY, {\mathbb G}_m) \cong \rmH^1_{Zar}(\rmY, {\mathbb G}_m) &\cong {\rm CH}^1(\rmY, 0) \mbox { and }\\
 \rmH^0_{et}(\rmY, {\mathbb G}_m) \cong \rmH^0_{Zar}(\rmY, {\mathbb G}_m) &\cong \Gamma(\rmY, \O_{\rmY})^* \cong {\rm CH}^1(\rmY, 1). \notag
\end{align}
Therefore, the map denoted $\alpha $ ($\beta$) in the  long exact sequence ~\eqref{local.seq} identifies
with the restriction 
\[{\rm CH}^1({\mathbb A}^m \times_{\rm Spec \, k}{\rm Spec \, A}, 1) \ra {\rm CH}^1(\rmU_m \times_{\rm Spec \, k}{\rm Spec \, A}, 1)\]
\[({\rm CH}^1({\mathbb A}^m \times_{\rm Spec \, k}{\rm Spec \, A}, 0) \ra {\rm CH}^1(\rmU_m \times_{\rm Spec \, k}{\rm Spec \, A}, 0), \res)\]
forming  part of the localization sequence for the higher Chow groups. In fact, the corresponding localization sequence is
given by:
\begin{multline}
 \label{local.seq.1}
   \begin{split}
\cdots \ra {\rm CH}^{1-c}({\rm Z_m\times_{Spec \, k}Spec \, A}, 1) \ra {\rm CH}^1({\rm {\mathbb A}^m \times_{Spec \, k}Spec \, A}, 1) {\overset {\alpha '} \ra} {\rm CH}^1(\rmU_m \times_{Spec \, k}Spec \, A, 1)\\
\ra {\rm CH}^{1-c}({\rm Z_m\times_{Spec \, k}Spec \, A}, 0) \ra {\rm CH}^1({\rm {\mathbb A}^m \times_{Spec \, k}Spec \, A}, 0) {\overset {\beta'} \ra} {\rm CH}^1(\rmU_m \times_{Spec \, k}Spec \, A, 0) \ra 0
   \end{split}
\end{multline}
where $c$ denotes
the codimension of $\rmZ_m$ in ${\mathbb A}^m$, which we assume is large. To see that one gets such a localization sequence,
one first replaces the strict Hensel ring $\rmA$ by one of the $\rmA_i$, where $\rmA =\lim_i \rmA_i$, with each $\rmA_i$ a regular local ring.
Clearly then the corresponding localization sequence exists and the groups in ~\eqref{local.seq.1} involving the $\rmZ_m$ are trivial, as $c$ is assumed to be large. 
At this point, one takes the direct limit over the $\rmA_i$: since
 the Chow groups are contravariantly functorial for flat maps, and filtered colimits are exact, we obtain the localization sequence ~\eqref{local.seq.1}.
Moreover, the  groups appearing in ~\eqref{local.seq.1} involving the ${\rm Z}_m$ are all trivial, thereby 
proving that the maps $\alpha'$ and $\beta'$ in ~\eqref{local.seq.1}, and therefore, the maps $\alpha$ and $\beta$ in ~\eqref{local.seq} are isomorphisms. This
provides the isomorphisms for $t=0, 1$:
\begin{align}
 \label{stalk.isoms}
 {\rm R}^t \rmp_{1*}({\mathbb G}_m)_{Spec \, A} \cong \rmH^t_{et}(\rmU_m \times_{Spec \, k}Spec \, A, {\mathbb G}_m) &\cong \rmH^t_{et}({\mathbb A}^m \times_{Spec \, k}Spec \, A, {\mathbb G}_m) \\
 &\cong \rmH^t_{et}(Spec \, A, {\mathbb G}_m). \notag
\end{align}
Therefore, it follows that for $t=1$, ${\rm R}^t \rmp_{1*}({\mathbb G}_m)_{Spec \, A} \cong 0$.
Since ${\mathbb G}_m$ is a sheaf on the flat and hence on the smooth topology, there is a natural map 
${\mathbb G}_m \ra \rmp_{1*}({\mathbb G}_m)$ of sheaves where the ${\mathbb G}_m$ on the left (on the right) is
a sheaf on the \'etale site of ${\EG \times_ {\rmG} \rmX}$ (on the \'etale site of 
$\EG \times_{\rm G}({\EG^{gm, {\it m}} \times\rmX})$, \res). The stalk-wise isomorphism in ~\eqref{stalk.isoms} for $t=0$ shows
that the natural map ${\mathbb G}_m \ra \rmp_{1*} ({\mathbb G}_m)$ of sheaves on the \'etale site is an isomorphism.
This provides the isomorphism:
\begin{equation}
\label{E2.p1}
\rmE_2^{1,0}(1) = \rmH^1_{et}(\EG\times_{\rmG}\rmX, \rmp_{1*}({\mathbb G}_m)) \cong \rmH^1_{et}(\EG\times_{\rmG}\rmX, {\mathbb G}_m).
\end{equation}
\vskip .1cm
Moreover, observing that the differentials in the spectral sequences above go from $\rmE_r^{p, q}$ to $\rmE_r^{p+r, q-r+1}$,
one sees that 
\begin{equation}
\rmE_r^{0,1}(1) = \rmE_r^{0,1}(2) =0 \mbox{ for all } r \ge 2 \mbox{ and that } \rmE_2^{1,0}(i) \cong E_r^{1,0}(i), \mbox{ for all } r\ge 2, i=1,2.
\end{equation}
The last observation shows that $\rmE_2^{1,0}(i)$, $i=1,2$ is isomorphic to the abutment in degree 1, namely,
$\rmH^1_{et}(\EG \times_{\rmG} (\EG^{gm, {\it m}} \times \rmX), {\mathbb G}_m)$, $m>>0$. Therefore, the isomorphisms 
in ~\eqref{E2.p1} and ~\eqref{E2.p2} complete the proof of (i).
\vskip .1cm
Next we consider the proof of (ii). The key point is to consider the Leray spectral sequences for the maps $p_1$ and $p_2$.
In this case, one may readily compute the stalks of $\rmR^tp_{i*}(\mu_{\ell^n}(1))$ to be trivial for $t= 1,2 $ and 
$\cong \mu_{\ell^n}(1)$ for $t=0$, and for $m>>0$. (See \cite[Theorem 1.6]{J20} for further details.) Therefore, the conclusions
 in (ii) follow readily.
\end{proof}
\begin{corollary}
\label{intr.stack}
 Assume the above context. 
 \vskip .1cm
(i) Then we obtain an isomorphism 
\[\rmH^1_{et}(\EG^{gm,m}\times_{\rmG}\rmX, {\mathbb G}_m) \cong \rmH_{et}^1(\EG\times_{\rmG}\rmX, {\mathbb G}_m) \cong \rmH_{smt}^1([\rmX/\rmG], {\mathbb G}_m), \mbox{ for } m >>0,\]
 which is functorial in the $\rmG$-scheme $\rmX$. 
 \vskip .1cm
 (ii) Moreover, we obtain isomorphisms:
 \[\rmH^2_{et}(\EG^{gm,m}\times_{\rmG}\rmX, \mu_{\ell^n}(1)) \cong \rmH_{et}^2(\EG\times_{\rmG}\rmX, \mu_{\ell^n}(1)) \cong \rmH_{smt}^2([\rmX/\rmG], \mu_{\ell^n}(1)) \mbox{ for } m >>0.\]
  which are functorial in the $\rmG$-scheme $\rmX$, and where $\ell \ne char (k)$. 
\vskip .1cm \noindent
  Here $\rmH_{smt}^1([\rmX/\rmG], {\mathbb G}_m)$ and $\rmH_{smt}^2([\rmX/\rmG], \mu_{\ell^n}(1))$ denote
 the cohomology of the quotient stack $[\rmX/\rmG]$ computed on the smooth site.
 \vskip .1cm
 (iii) One obtains an isomorphism $\Br_{\rmG}(\rmX)_{\ell^n} \cong \Br([\rmX/\rmG])_{\ell^n}$, thereby proving that
 $\Br_{\rmG}(\rmX)_{\ell^n}$ is an invariant of the quotient stack $[\rmX/\rmG]$, for any 
 prime $\ell \ne char (k)$.
\end{corollary}
\begin{proof}
 The first isomorphisms in both the statements (i) and (ii) are from Proposition ~\ref{comp.isoms}.
 The second isomorphisms in (i) and (ii) follow from the isomorphism of the simplicial schemes: $\EG\times_{\rmG}\rmX \cong cosk_0^{[\rmX/\rmG]}(\rmX)$ and Proposition ~\ref{coh.comp} discussed below.
  Next we consider the
  third statement.
 \vskip .1cm
 Recall the long exact sequence in \'etale cohomology obtained from the Kummer sequence:
 \begin{multline}
  \begin{split}
\label{Kummer.seq.2}
\rightarrow \H^1_{et}(\EG\times_{\rmG}{\rmX},{\mathbb G}_m) {\overset {\ell^n} \ra } \H^1_{et}(\EG\times_{\rmG}{\rmX}, {\mathbb G}_m) {\overset {\delta} \ra}  \H^2_{et}(\EG\times_{\rmG}{\rmX},\mu_{\ell^n}(1)) \rightarrow \H^2_{et}(\EG\times_{\rmG}{\rmX}, {\mathbb G}_m) \\
{\overset {\ell ^n} \rightarrow} \H^2_{et}(\EG\times_{\rmG}{\rmX}, {\mathbb G}_m) \rightarrow \cdots 
\end{split}\end{multline}
 Then the $cokernel (\H^1_{et}(\EG\times_{\rmG}{\rmX},{\mathbb G}_m) {\overset {\ell^n} \ra } \H^1_{et}(\EG\times_{\rmG}{\rmX}, {\mathbb G}_m))$ maps
 to $ \H^2_{et}(\EG\times_{\rmG}{\rmX},\mu_{\ell^n}(1))$, by a map induced by the boundary map $\delta$: we will denote this map by $\bar \delta$.
 Then, in view of the isomorphisms in (i) and (ii), the Brauer group $\Br_{\rmG}(\rmX)_{\ell^{\nu}}$ identifies with the cokernel of the map 
 $\bar \delta$.
 \vskip .1cm
 In view of Proposition ~\ref{coh.comp}, the isomorphisms in (i) and (ii) and 
 the long exact sequence ~\eqref{Kummer.seq.2}, $\Br([\rmX/\rmG])_{\ell^n}$ 
 identifies with 
 \[kernel (\rmH^2_{et}(\EG\times_{\rmG}\rmX, {\mathbb G}_m) {\overset {\ell^n} \ra} \rmH^2_{et}(\EG\times_{\rmG}\rmX, {\mathbb G}_m)).\]
 Again by Proposition ~\ref{coh.comp}, the isomorphisms in (i) and (ii) and 
 the long-exact sequence ~\eqref{Kummer.seq.2}, this identifies with 
 \[cokernel((\H^1_{et}(\EG\times_{\rmG}\rmX,{\mathbb G}_m)/\ell^n {\overset {\bar \delta} \ra } \H^2_{et}(\EG\times_{\rmG}\rmX, \mu_{\ell^n}(1))) \cong \Br_{\rmG}(\rmX)_{\ell^n}.\]
 This proves the third assertion and hence
  Theorem ~\ref{Br.grp.quot.stacks}.
\end{proof}
Let ${\it S}$ denote an algebraic stack, which we will assume is of {\it Artin type} and of finite type over the given base field $k$, 
with $x: \rmX \ra {\it S}$ an {\it atlas}, that is,
a {\it smooth surjective map from an algebraic space $\rmX$}. We let $\rmB_{\it x}{\it S} = cosk_0^{\it S}(\rmX)$ denote the corresponding simplicial
algebraic space. Then we let ${\it S}_{smt}$ denote the smooth site, whose objects are $y:\rmY \ra {\it S}$, with $y$ a smooth map from
an algebraic space $\rmY$ to ${\it S}$, and where a morphism between two such objects $y': \rmY' \ra {\it S}$ and $y:\rmY \ra {\it S}$ is
given by a map $f: \rmY' \ra \rmY$ making the triangle
\[ \xymatrix{{\rmY'} \ar@<1ex>[rr]^f \ar@<1ex>[dr]_{y'} && {\rmY} \ar@<-1ex>[dl]^{y}\\
                                &{\it S}}
\]
commute. The same definition defines the smooth site of any algebraic space. The smooth and \'etale sites of the
simplicial algebraic space $\rmB_{\it x}{\it S}$ may be defined as follows. The objects of $Smt(\rmB_{\it x}{\it S})$ are given by smooth maps
$u_n:\rmU_n \ra (\rmB_{\it x}{\it S})_n$ for some $n \ge 0$. Given such a $u_n: \rm U_n \ra \rmB_{\it x}{\it S}_n$ and $v_m: \rmV_m \ra \rmB_{\it x}S_m$,
a morphism from $u_n \ra v_m$ is a commutative square:
\vskip .1cm
\[\xymatrix{{\rmU_n} \ar@<1ex>[r]^{\alpha'} \ar@<1ex>[d]^{u_n} & {\rmV_m} \ar@<1ex>[d]^{v_m}\\
           {\rmB_{\it x}{\it S}_n} \ar@<1ex>[r] ^{\alpha} & {\rmB_{\it x}{\it S}_m}}
\]
where $\alpha$ is a structure map of the simplicial algebraic space $\rmB_{\it x}{\it S}$. The \'Etale site $Et(\rmB_{\it x}{\it S})$ is defined
similarly. An abelian sheaf $F$ on $Smt(\rmB_{\it x}{\it S})$  is given by a collection of abelian sheaves $F=\{F_n|n\}$ with each
$F_n$ being an abelian sheaf on $Smt(\rmB_{\it x}{\it S}_n)$, so that it comes equipped with the following data: for each 
structure map $\alpha: \rmB_{\it x}{\it S}_n \ra \rmB_{\it x}{\it S}_m$, one is provided with a map of sheaves $\phi_{n, m}:\alpha^*(F_m) \ra F_n$ so that 
the maps $\{\phi_{n,m}|n,m\}$ are compatible. Abelian sheaves on the site $Et(\rmB_{\it x}{\it S})$ may be defined similarly.
We skip the verification that the category of abelian sheaves on the above sites have enough injectives. 
The $n$-th cohomology group of the simplicial object $\rmB_{\it x}{\it S}$ with respect to an abelian sheaf $F$ is defined as the $n$-th right derived
functor of the functor sending
\begin{equation}
\label{coh.1}
F \mapsto kernel( \delta^0 - \delta^1: \Gamma(\rmB_{\it x}{\it S}_0, F_0) \ra \Gamma (\rmB_{\it x}{\it S}_1, F_1)).
\end{equation}
\vskip .1cm
Now we obtain the following Proposition.
\begin{proposition}
\label{coh.comp}
 Let $F$ denote an abelian sheaf on $Smt(S)$. Then we obtain the following isomorphisms:
 \vskip .1cm
 (i) $\rmH^*_{smt}(\rmB_{\it x}{\it S}, {\it x}_{\bullet}^*(F)) \cong \rmH^*_{smt}({\it S}, F)$, where the subscript $smt$ denotes cohomology computed on
 the smooth sites and $x_{\bullet}: \rmB_{\it x}{\it S} \ra {\it S}$ is the simplicial map induced by $x: \rmX \ra {\it S}$.
 \vskip .1cm
 (ii) $\rmH^*_{smt}(\rmB_{\it x}{\it S}, {\it x}_{\bullet}^*(F)) \cong \rmH^*_{et}(\rmB_{\it x}{\it S}, \alpha_*{\it x}_{\bullet}^*(F))$, where the subscript $et$ denotes
 cohomology computed on the \'etale site and $\alpha: Smt(\rmB_{\it x}{\it S}) \ra Et(\rmB_{\it x}{\it S})$ is the obvious morphism of sites.
\end{proposition}
\begin{proof}
 Observe that $x:\rmX \ra {\it S}$ is a covering of the stack $S$ in the smooth topology, so that 
 \[kernel (\delta^0 - \delta ^1: \Gamma(\rmB_{\it x}{\it S}_0, F_0) \ra \Gamma(\rmB_{\it x}{\it S}_1, F_1)) \cong \Gamma ({\it S}, F).\]
 Since $\rmH^n_{smt}({\it S}, F)$ is the $n$-th right derived functor of the above functor, in view of ~\eqref{coh.1}, we see that it identifies with
 $\rmH^n_{smt}(\rmB_{\it x}{\it S}, {\it x}_{\bullet}^*(F))$. This provides the isomorphism in (i).
 The isomorphism in (ii) is a straight-forward extension of a well-known result comparing the cohomology of an algebraic space computed on the smooth and \'etale sites.
\end{proof}

\section{Brauer groups of algebraic stacks: Proofs of Theorems ~\ref{nonequiv.to.equiv} through ~\ref{main.thm.2}}

\vskip .1cm \noindent
{\bf Proof of Theorem ~\ref{nonequiv.to.equiv}.}
\vskip .1cm

We  proved a related result in  \cite[Corollary 1.16(ii)]{JP20}, which is weaker than the above theorem in the following sense:
in \cite[Corollary 1.16(ii)]{JP20}, we assumed the cycle map for $\rmX$ is an isomorphism in all degrees, whereas, here we are
considering the more general case, where it is assumed to be an isomorphism only in degrees less than or equal to $2$. Therefore, we will
provide full details of the proof.
However, we also make strong use of the motivic and \'etale Becker-Gottlieb transfer due to the second author and Carlsson: see \cite{CJ20}.
 Invoking this transfer and \cite[Corollary 1.15]{JP20}, one observes the isomorphisms, assuming $|\rmW|$ is relatively prime to $\ell$:
\begin{align}
\label{G.to.T}
 \rmH^{*, \bullet}_{\rmG, \M}(X, Z/\ell^n) &\cong \rmH^{*, \bullet}_{\rmT, \M}(X, Z/\ell^n)^{\rmW} \mbox{and}\\
\rmH^{*, \bullet}_{\rmG, et}(X, Z/\ell^n) &\cong \rmH^{*, \bullet}_{\rmT, et}(X, Z/\ell^n)^{\rmW} \notag.
\end{align}
Therefore, we reduce to the case where $\rmG$ is replaced by a split torus $\rmT$. At this point, we observe that a choice of 
$\BT^{gm,m} = \Pi_{i=1}^n{\mathbb P}^m $, if $\rmT = {\mathbb G}_m ^n$.  
\vskip .1cm
In the above discussion, it is important to restrict to
 linear algebraic groups $\rmG$ that are special, so that the construction of the transfer can be carried out on the 
 Nisnevich site using the geometric Borel construction  $\{\EG^{gm,m}\times_{\rmG}\rmX|m \ge 0\}$
 as discussed in ~\eqref{geom.Borel}. If the group $\rmG$ is not special, then the construction of the transfer
 has to be carried out on the \'etale site using the same gadgets followed by a derived push-forward to the Nisnevich site.
 This will then not give the same Brauer group for the stack as discussed in Definition ~\ref{equiv. Brauer.grp}: see also
 Corollary ~\ref{intr.stack}.
\vskip .2cm
In view of the above reduction to the case where the group $\rmG$ is a split torus, 
following discussion now proves both statements (i) and (ii).
Observe that $\ET^{gm,m} \ra \BT^{gm,m}$ is a Zariski locally trivial torsor for the action $\rmT$ as $\rmT = {\mathbb G}_m^n$ is a
split torus, and hence is special as a linear algebraic group in the sense of Grothendieck: see \cite{Ch}. 
Taking $n=1$, we see that $\pi^m: \rmE{\mathbb G}_m^{gm,m} \ra \rmB{\mathbb G}_m^{gm,m}  $ is such a torsor, so that there is a Zariski open cover 
$\{\rmU_j|j=1, \cdots, \rmN\}$
so that $\pi^m_{|\rmU_j}$ is of the form $\rmU_j \times {\mathbb G}_m^n \ra \rmU_j$, $j=1, \cdots, \rmN$.
 \vskip .2cm
 
Let $\{\rmV_0, \cdots \rmV_m\}$ denote the open cover of ${\mathbb P}^m$ obtained by letting $\rmV_i$ be the open subscheme where 
the homogeneous coordinates $x_i$, ($i=0,\cdots, m$) on ${\mathbb P}^m$ is non-zero.
Without loss of 
generality, we may assume the $\rmU_j$ refine the open cover $\{\rmV_i|i=0, \cdots m\}$. Finally the observation that the Picard groups of
 affine spaces are trivial, shows that one may in fact take $\rmN=m$ and $\rmU_j= \rmV_j$, $j=0, \cdots, m$.
Now one may take an open cover of $\Pi_{i=1}^n{\mathbb P}^m$ by taking the product of the affine spaces that form the open cover of
 each factor ${\mathbb P}^m$. We will denote this open cover of $\Pi_{i=1}^n {\mathbb P}^m$ by $\{\rmW_{\alpha}|\alpha\}$.
\vskip .2cm
Let $p: \ET^{gm,m} {\underset {\rmT} \times}\rmX \ra \BT^{gm,m}$ denote the obvious map, and let $\epsilon$ denote the
map from the \'etale site to the Nisnevich site. For any integer $j \ge 0$, let ${\mathbb Z}/\ell^n(j)$ denote the motivic complex of weight $j$
on the Nisnevich site of $\ET^{gm,m}{\underset {\rmT} \times}X$. Then one obtains the identification (see \cite{Voev11} or \cite{HW}):
\begin{equation}
   {\mathbb Z}/\ell^n(j) = \tau_{\le j} \rmR \epsilon_* \epsilon ^*({\mathbb Z}/\ell^n(j)).  
\end{equation} 
\vskip .1cm
Therefore, on applying $\rmR p_*$, we obtain the natural maps:
\begin{align}
\label{cycl.map}
\rmR p_*({\mathbb Z}/\ell^n(j)) {\overset {\simeq} \ra} \rmR p_*(\tau_{\le j} \rmR\epsilon_* \epsilon ^*({\mathbb Z}/\ell^n(j))) &\ra \rmR p_* R\epsilon_* \epsilon^*({\mathbb Z}/\ell^n(j))\\
&\cong \rmR p_*\rmR \epsilon_*\mu_{\ell^n}(j)) \cong \rmR \epsilon_* \rmR p_* \mu_{\ell^n}(j)).\notag
\end{align} 
Next we make the following key observations:
\begin{enumerate}[\rm(i)]
\label{key.observs}
 \item On taking sections over any Zariski open subset $\rmU$ of $\BT^{gm,m}$, and cohomology in degrees $u \le v$, we obtain 
 an isomorphism: $\rmH^{\it u}_M(p^{-1}(\rmU), {\mathbb Z}/\ell^n({\it v})) {\overset {\cong} \ra} \rmH^{\it u}_{et}(p^{-1}(\rmU), \mu_{\ell^n}({\it v}))$
 where $p^{-1}(\rmU) = (\ET^{gm,m} \times_{\rmT}X){\underset {\BT^{gm,m} } \times}\rmU$. This should be clear in view of the
 map of spectral sequences:
 \begin{align}
 \label{key.observ.1}
\rmE_2^{s,t} = \rmH^s_{Nis}(U, \rmR^tp_*(\tau_{\le j} \rmR\epsilon_* \epsilon ^*({\mathbb Z}/\ell^n(j)))) & \Rightarrow \rmH^{s+t}_{Nis}(p^{-1}(U), {\mathbb Z}/\ell^n(j))\\
\rmE_2^{s,t} = \rmH^s_{Nis}(U, \rmR^tp_*( \rmR\epsilon_* \epsilon ^*({\mathbb Z}/\ell^n(j)))) & \Rightarrow \rmH^{s+t}_{Nis}(p^{-1}(U),  \rmR\epsilon_* \epsilon ^*({\mathbb Z}/\ell^n(j)))\cong \rmH^{s+t}_{et}(p^{-1}(U), \mu_{\ell^n}(j)) \notag
\end{align} 
which induces an isomorphism on the $\rmE_2$-terms for $0 \le s+t \le j$, as $s \ge 0$. 
 \item On taking the sections over each Zariski open set in the above cover $\rmW_{\alpha}$ of $\BT^{gm,m}$, we obtain a quasi-isomorphism, since affine-spaces
   are contractible for motivic cohomology and also \'etale cohomology with respect to $\mu_{\ell^n}(j)$, with $\ell \ne char (k)$, as $k$ is assumed to be 
   separably closed.
\end{enumerate}
Now a Mayer-Vietoris argument using the above open cover of $\BT^{gm,m}$ completes the proof. Since the iterated
    intersections of the affine spaces forming the open cover of $\BT^{gm,m}$ are products of an affine space with a split torus,
      one may invoke Proposition ~\ref{inductive.pf} to complete the proof of statements (i) and (ii) in the theorem.
\vskip .1cm
In order to prove the statement (iii) when $\rmX=\rmG/\rmH$, we first observe that $\EG^{gm,m}\times_{\rmG} \rmG/\rmH$ identifies with $\BH^{gm,m}$, $m>>0$.
Therefore, statement (ii) follows from Proposition ~\ref{torsion.index.prop} proven below,  once one identifies $\EG^{gm,m}\times_{\rmG}(\rmG/\rmH)$ with
  ${\rm BH}^{\rm gm,m}$.
\qed
\vskip .2cm \noindent
{\bf Proof of Corollary ~\ref{cor.1}.} We will prove (i) under the assumption the base field is separably closed.
Then, we obtain the short exact sequences from the Kummer sequence:
\begin{equation}
 \label{Kummer.seq.3}
0 \ra \Pic({\rmX})/\ell^n \ra \rmH^2_{et}({\rmX}, \mu_{\ell^n}(1)) \ra \Br({\rmX})_{\ell^n} \ra 0 \, \notag \mbox{ and }
\end{equation} 
\begin{equation}
 \label{Kummer.seq.4}
0 \ra \Pic(\EG^{gm,m}{\underset {\rmG} \times}{\rmX})/\ell^n \ra \rmH^2_{et}(\EG^{gm,m}{\underset {\rmG} \times}{\rmX}, \mu_{\ell^n}(1)) \ra \Br(\EG^{gm,m}{\underset {\rmG} \times}{\rmX})_{\ell^n} = \Br_{\rmG}({\rmX})_{\ell^n}\ra 0. \notag 
\end{equation} 
Now it suffices to observe that the maps: 
\[\Pic({\rmX})/\ell^n \cong \rmH^{2,1}_M(\rmX, {\mathbb Z}/\ell^n) \ra \rmH^2_{et}({\rmX}, \mu_{\ell^n}(1)) \mbox { and }\]
\[\Pic(\EG^{gm,m}{\underset {\rmG} \times}{\rmX})/\ell^n \cong \rmH^{2,1}_M(\EG^{gm,m}{\underset {\rmG} \times}\rmX, {\mathbb Z}/\ell^n)\ra \rmH^2_{et}(\EG^{gm,m}{\underset {\rmG} \times}{\rmX}, \mu_{\ell^n}(1)) \]
identify with the
cycle maps
\[\rmH^{2, 1}_{\M}(\rmX, Z/\ell^n) \ra \rmH^2_{et}(\rmX, \mu_{\ell^n}(1)) \mbox { and } \rmH^{2,1}_{\rmG, \M}(\rmX, Z/\ell^n) \ra \rmH^2_{\rmG, et}(\rmX, \mu_{\ell^n}(1))\]
which are both injective, with the cokernel of the first map (the second map) given by $\Br(\rmX)_{\ell^n}$ ($\Br_{\rmG}(\rmX)_{\ell^n}$, \res).
Therefore, the triviality of $\Br(\rmX)_{\ell^n}$ ($\Br_{\rmG}(\rmX)_{\ell^n}$) is equivalent to the first cycle map 
(the second cycle map, \res) being an isomorphism. Therefore, the corollary follows when the base field is separably closed.
The extension to the case when it is not is now clear, since we are still only considering the Brauer group $\Br([\rmX^s/\rmG^s])_{\ell^n}$. 
This completes the proof of (i).
\vskip .1cm 
The statement in (ii) now follows readily from (i), in view of Theorem ~\ref{Br.grp.quot.stacks}. \qed
\vskip .2cm \noindent
{\bf Proof of Theorem ~\ref{nonequiv.to.equiv.1}.} This follows immediately from Theorem ~\ref{nonequiv.to.equiv}, once one observes
the isomorphisms:
\[\rmH^{*, \bullet}_{\rmG, \M}(\rmX, Z/\ell^n) \cong  \rmH^{*, \bullet}_{\tilde \rmG, \M}(\tilde \rmG \times_{\rmG} \rmX, Z/\ell^n) \mbox{ and }\]
\[\rmH^{*}_{\rmG, et}(\rmX, \mu_{\ell^n}(\bullet)) \cong \rmH^{*}_{\tilde \rmG, et}(\tilde \rmG \times_{\rmG} \rmX, \mu_{\ell^n}(\bullet)).\]
\vskip .2cm \noindent
{\bf Proof of Theorem ~\ref{main.thm.2}.} 
We first observe that ${\rm B}{\mathbb G}_m^{\rm gm} \cong \colimn {\mathbb P}^n$. Since each ${\mathbb P}^n$ is a linear scheme which is
projective and smooth, it follows from \cite[Theorem 4.5, Corollary 4.6]{J01} that one obtains isomorphisms for any smooth scheme $\rmY$:
\begin{align}
 \label{kunneth}
\oplus_i \H_{M}^{2i, i}({\rm B}{\mathbb G}_m^{\rm gm} \times \rmY, {\mathbb Z}/\ell^n) &\cong (\oplus _i \H_{M}^{2i, i}({\rm B}{\mathbb G}_m, {\mathbb Z}/\ell^n)) \otimes (\oplus _i\H_{M}^{2i, i}(\rmY, {\mathbb Z}/\ell^n)) \mbox{ and }\\
\oplus_i \H_{et}^{2i}({\rm B}{\mathbb G}_m^{\rm gm} \times \rmY, \mu_{\ell^n}(i)) &\cong (\oplus _i \H_{et}^{2i}({\rm B}{\mathbb G}_m, \mu_{\ell^n}(i))) \otimes (\oplus _i\H_{et}^{2i}(\rmY, \mu_{\ell^n}(i))). \notag
\end{align}
Since the cycle map $cycl: \oplus _i\H_{M}^{2i, i}({\rm B}{\mathbb G}_m^{\rm gm}, {\mathbb Z}/\ell^n) \ra \oplus _i\H_{et}^{2i}({\rm B}{\mathbb G}_m^{\rm gm}, \mu_{\ell^n}(i))$ is an isomorphism, the Brauer group
$\Br({\rm Bun}_{1, d}(\rmX))_{\ell ^n}$, which is the cokernel of cycle map, identifies with $\Br({\bold Pic}^d(\rmX))_{\ell ^n}$. Finally the
isomorphism $\Br({\bold Pic}^d(\rmX))_{\ell ^n} \cong \Br({\rm Sym}^d(\rmX))_{\ell ^n}$ is proven in \cite[Theorem 1.2]{IJ20}.
\vskip .2cm
Recall that $\Br(\rmY)=0$ if $\rmY$ is a connected projective smooth variety that is {\it rational}: this follows from the well-known fact that the Brauer group
is a stable birational invariant for connected projective smooth varieties. The last statement follows from this observation.
\qed
\vskip .2cm
\begin{lemma}
 \label{isom.lemma}
 Let $p: \cX \ra \cY$ denote a map of smooth  schemes over $k$, so that it is Zariski locally trivial, with fibers given by the
  scheme $\rmX$ satisfying the condition that the cycle map:
  \[cycl: \rmH^{2,1}_M(X, {\mathbb Z}/\ell^n) \ra \rmH^{2}_{et}(X, \mu_{\ell^n}(1))\]
  is an isomorphism. Let $\rmU$, $\rmV$ denote two Zariski open subschemes of $\cY$ so that $\cX\times_{\cY}\rmU \cong \rmU \times \rmX$
  and $\cX\times_{\cY}\rmV \cong \rmV \times \rmX$. Assume that the corresponding cycle maps
  \[\rmH^{2,1}_M(\cX\times_{\cY}\rmU, {\mathbb Z}/\ell^n) \ra \rmH^{2}_{et}(\cX\times_{\cY}\rmU, \mu_{\ell^n}(1)) \mbox{ and }\rmH^{2,1}_M(\cX\times_{\cY}\rmV, {\mathbb Z}/\ell^n) \ra \rmH^{2}_{et}(\cX\times_{\cY}\rmV, \mu_{\ell^n}(1))\]
  are both isomorphisms and the cycle map
  \[\rmH^{2,1}_M(\cX\times_{\cY}(\rmU \cap \rmV), {\mathbb Z}/\ell^n) \ra \rmH^{2}_{et}(\cX\times_{\cY}(\rmU\cap \rmV), \mu_{\ell^n}(1))\]
  is a monomorphism. Then the cycle map
  \[\rmH^{2,1}_M(\cX\times_{\cY}(\rmU \cup \rmV), {\mathbb Z}/\ell^n) \ra \rmH^{2}_{et}(\cX\times_{\cY}(\rmU\cup \rmV), \mu_{\ell^n}(1))\]
is an isomorphism.
\end{lemma}
\begin{proof}
For a subscheme $\rmW$ in $\rmY$, we will continue to let $\cX_{W} = \cX\times_{\cY}\rmW$.
Now we consider the commutative diagram with exact rows:
\[\xymatrix{{\rmH^{1,1}_{M}(\cX_{\rmU}, {\mathbb Z}/\ell^n) \oplus \rmH^{1,1}_{M}(\cX_{\rmV}, {\mathbb Z}/\ell^n)} \ar@<1ex>[r] \ar@<1ex>[d] & {\rmH^{1, 1}_M(\cX_{\rmU \cap \rmV},{\mathbb Z}/\ell^n)} \ar@<1ex>[r]  \ar@<1ex>[d]& {\rmH^{2, 1}_M(\cX_{\rmU \cup \rmV}, {\mathbb Z}/\ell^n)} \ar@<1ex>[d]\\
  {\rmH^{1}_{et}(\cX_{\rmU}, \mu_{\ell^n}(1)) \oplus \rmH^{1}_{et}(\cX_{\rmV}, \mu_{\ell^n}(1))} \ar@<1ex>[r] & {\rmH^{1}_{et}(\cX_{\rmU \cap \rmV}, \mu_{\ell^n}(1))} \ar@<1ex>[r] & {\rmH^{2}_{et}(\cX_{\rmU \cup \rmV}, \mu_{\ell^n}(1)) }}
 \]
\[\xymatrix{{} \ar@<1ex>[r]  &  {\rmH^{2,1}_{M}(\cX_{\rmU}, {\mathbb Z}/\ell^n) \oplus \rmH^{2,1}_{M}(\cX_{\rmV}, {\mathbb Z}/\ell^n)} \ar@<1ex>[d] \ar@<1ex>[r] &  {\rmH^{2, 1}_M(\cX_{\rmU \cap \rmV}, {\mathbb Z}/\ell^n)} \ar@<1ex>[d]\\
  {}  \ar@<1ex>[r] & {\rmH^{2}_{et}(\cX_{\rmU}, \mu_{\ell^n}(1)) \oplus \rmH^{2,1}_{et}(\cX_{\rmV}, \mu_{\ell^n}(1))}  \ar@<1ex>[r] & {\rmH^{2}_{et}(\cX_{\rmU \cap \rmV}, \mu_{\ell^n}(1))} }
 \]
 In view of the spectral sequence in ~\eqref{key.observ.1} with $j=1$, one may observe that the second vertical map is an isomorphism.
 Therefore, Lemma ~\ref{diagm.chase.lemma}(ii) applies to prove the required map is an isomorphism
\end{proof}

\begin{proposition}
 \label{inductive.pf}
 Let $p:\cX \ra \cY$ denote a map of smooth schemes over $k$, satisfying the hypotheses of Lemma ~\ref{isom.lemma}. We will further assume  the following:
 let $\rmU_i, i=1, \cdots, n$ denote open subsets of $\cY$, so that the hypotheses of Lemma ~\ref{isom.lemma} holds
 with $\rmU$, $\rmV$ denoting any two of these open sets. Assume further that there exists an affine space ${\mathbb A}^N$
 so that each $\rmU_i \cong {\mathbb A}^N$ and that each intersection $\rmU_{i_1} \cap \rmU_{i_2} \cong {\mathbb G}_m \times {\mathbb A}^{N-1}$.
 Then the following holds, where for a subscheme $\rmW$ in $\rmY$, we will let $\cX_{W} = \cX\times_{\cY}\rmW$, and  cycl will denote the higher cycle map:
 \begin{enumerate}[\rm(i)]
  \item  $cycl: \rmH^{2,1}_M(\cX_{(\rmU_1 \cup \cdots \cup \rmU_{n-1}) \cap \rmU_n}, {\mathbb Z}/\ell^n) \ra \rmH^2_{et}(\cX_{(\rmU_1 \cup \cdots \cup\rmU_{n-1}) \cap \rmU_n}, \mu_{\ell^n}(1))$
  is a monomorphism \mbox { and }
  \item $cycl: \rmH^{2,1}_M(\cX_{(\rmU_1 \cup \cdots \cup \rmU_{n-1}) \cup \rmU_n}, {\mathbb Z}/\ell^n) \ra \rmH^2_{et}(\cX_{(\rmU_1 \cup \cdots \cup \rmU_{n-1}) \cup \rmU_n}, \mu_{\ell^n}(1))$ is an isomorphism.
 \end{enumerate}
\end{proposition}
\begin{proof}
 We will prove these using ascending induction on $n$.  Observe that the case $n=2$ is handled by 
 Lemma ~\ref{1.intersect}. We will first consider (i).

 Assume next that (i) holds when $\rmU_i$, $i=1, \cdots, n$ are any open subsets of $\cY$ satisfying the hypotheses.
 Let $\rmU_i$, $i=1, \cdots, n, n+1$ be open subsets satisfying the hypotheses. Let $\rmW_1 = (\rmU_1 \cup \cdots \cup  \rmU_{n-1}) \cap \rmU_{n+1}$
 and let $\rmW_2 = \rmU_n \cap \rmU_{n+1}$. Then we obtain the commutative diagram:
 \[\xymatrix{{\rmH^{1,1}_{M}(\cX_{\rmW_1}, {\mathbb Z}/\ell^n) \oplus \rmH^{1,1}_{M}(\cX_{\rmW_2}, {\mathbb Z}/\ell^n)} \ar@<1ex>[r] \ar@<1ex>[d] & {\rmH^{1, 1}_M(\cX_{\rmW_1 \cap \rmW_2},{\mathbb Z}/\ell^n)} \ar@<1ex>[r]  \ar@<1ex>[d]& {\rmH^{2, 1}_M(\cX_{\rmW_1 \cup \rmW_2}, {\mathbb Z}/\ell^n)} \ar@<1ex>[d]\\
  {\rmH^{1}_{et}(\cX_{\rmW_1}, \mu_{\ell^n}(1)) \oplus \rmH^{1}_{et}(\cX_{\rmW_2}, \mu_{\ell^n}(1))} \ar@<1ex>[r] & {\rmH^{1}_{et}(\cX_{\rmW_1 \cap \rmW_2}, \mu_{\ell^n}(1))} \ar@<1ex>[r] & {\rmH^{2}_{et}(\cX_{\rmW_1 \cup \rmW_2}, \mu_{\ell^n}(1)) }}
 \]
\[\xymatrix{{} \ar@<1ex>[r]  &  {\rmH^{2,1}_{M}(\cX_{\rmW_1}, {\mathbb Z}/\ell^n) \oplus \rmH^{2,1}_{M}(\cX_{\rmW_2}, {\mathbb Z}/\ell^n)} \ar@<1ex>[d] \ar@<1ex>[r] &  {\rmH^{2, 1}_M(\cX_{\rmW_1 \cap \rmW_2}, {\mathbb Z}/\ell^n)} \ar@<1ex>[d]\\
  {}  \ar@<1ex>[r] & {\rmH^{2}_{et}(\cX_{\rmW_1}, \mu_{\ell^n}(1)) \oplus \rmH^{2,1}_{et}(\cX_{\rmW_2}, \mu_{\ell^n}(1))}  \ar@<1ex>[r] & {\rmH^{2}_{et}(\cX_{\rmW_1 \cap \rmW_2}, \mu_{\ell^n}(1))} }
 \]
 Then the inductive assumption, together with Lemma ~\ref{1.intersect} show the map $\rm\rmH^{2,1}_{M}(\cX_{\rmW_1}, {\mathbb Z}/\ell^n) \ra \rmH^{2}_{et}(\cX_{\rmW_1}, \mu_{\ell^n}(1)) $
 is a monomorphism while Lemma ~\ref{1.intersect} shows the map $\rm\rmH^{2,1}_{M}(\cX_{\rmW_2}, {\mathbb Z}/\ell^n) \ra \rmH^{2}_{et}(\cX_{\rmW_2}, \mu_{\ell^n}(1)) $ is a monomorphism.
 Observe that $\rmW_1 \cup \rmW_2 = (\rmU_1 \cup \cdots \cup \rmU_n) \cap \rmU_{n+1}$.
 In view of the spectral sequence in ~\eqref{key.observ.1} with $j=1$, one may observe that the first two vertical maps are isomorphisms.
 Therefore, now an application of Lemma ~\ref{diagm.chase.lemma}(i) then shows the cycle map $\rm\rmH^{2, 1}_M(\cX_{\rmW_1 \cup \rmW_2}, {\mathbb Z}/\ell^n) \ra \rmH^{2}_{et}(\cX_{\rmW_1 \cup \rmW_2}, \mu_{\ell^n}(1))$ is 
 a monomorphism, thereby completing the proof of (i).
 \vskip .1cm
 At this point (ii) follows readily from Lemma ~\ref{isom.lemma} by taking $\rmU = \rmU_1 \cup \cdots \cup \rmU_n$
 and $\rmV= \rmU_{n+1}$ there. Now observe that $\rmU \cap \rmV = (\rmU_1 \cup \cdots \cup \rmU_n) \cap \rmU_{n+1}$. (i) proved
 above shows that the cycle map
 \[\rmH^{2,1}_M(\cX\times_{\cY}(\rmU \cap \rmV), {\mathbb Z}/\ell^n) \ra \rmH^{2}_{et}(\cX\times_{\cY}(\rmU\cap \rmV), \mu_{\ell^n}(1))\]
  is a monomorphism. The inductive assumption now shows that 
  the cycle map
 \[\rmH^{2,1}_M(\cX\times_{\cY}\rmU, {\mathbb Z}/\ell^n) \ra \rmH^{2}_{et}(\cX\times_{\cY}\rmU, \mu_{\ell^n}(1))\]
  is an isomorphism. Therefore, the hypotheses of Lemma ~\ref{isom.lemma} are satisfied, so that Lemma ~\ref{isom.lemma} applies to complete the proof of (ii).
\end{proof}
\vskip .2cm

\begin{lemma} 
\label{1.intersect}
Assume that $\rmX$ is a smooth scheme so that the cycle map
\[cycl: \rmH^{i, 1}_M(X, {\mathbb Z}/\ell^n) \ra \rmH^i_{et}(X, \mu_{\ell^n}(1))\]
is an isomorphism for all $0 \le i \le 2$. Then 
the induced cycle map 
 $\rm\rmH^{i, 1}_M(\rmX \times {\mathbb G}_m, {\mathbb Z}/\ell^n) \ra \rmH^i_{et}(\rmX \times {\mathbb G}_m, \mu_{\ell^n}(1))$
is injective for  for all $0 \le i \le 2$.
\end{lemma}
\begin{proof} 
In view of the observation ~\eqref{key.observ.1} above, the above cycle map is an isomorphism for $i=0$ or $i=1$. Therefore, it suffices
to consider the case $i=2$.
  This follows from the commutative diagram 
 of localization sequences:
 \[\xymatrix{{\rmH^{2,1}_{\rmX \times \{0\} ,M}(\rmX \times {\mathbb A}^1, {\mathbb Z}/\ell^n)} \ar@<1ex>[r] \ar@<1ex>[d] & {\rmH^{2, 1}_M(\rmX \times {\mathbb A}^1,{\mathbb Z}/\ell^n)} \ar@<1ex>[r]  \ar@<1ex>[d]& {\rmH^{2, 1}_M(\rmX \times {\mathbb G}_m, {\mathbb Z}/\ell^n)} \ar@<1ex>[d]\\
  {\rmH^{2}_{\rmX \times \{0\}, et}(\rmX \times {\mathbb A}^1, \mu_{\ell^n}(1))} \ar@<1ex>[r] & {\rmH^{2}_{et}(\rmX \times {\mathbb A}^1, \mu_{\ell^n}(1))} \ar@<1ex>[r] & {\rmH^{2}_{et}(\rmX \times {\mathbb G}_m, \mu_{\ell^n}(1)) }}
 \]
\[\xymatrix{{} \ar@<1ex>[r]  & {\rmH^{3,1}_{\rmX \times \{0\}, M}(\rmX  \times {\mathbb A}^1, {\mathbb Z}/\ell^n)} \ar@<1ex>[d] \ar@<1ex>[r] & {\rmH^{3, 1}_M(\rmX \times {\mathbb A}^1, {\mathbb Z}/\ell^n)} \ar@<1ex>[d]\\
  {}  \ar@<1ex>[r] & {\rmH^{3}_{\rmX \times \{0\}, et}(\rmX  \times {\mathbb A}^1, \mu_{\ell^n}(1))}  \ar@<1ex>[r] & {\rmH^{3}_{et}(\rmX \times {\mathbb A}^1, \mu_{\ell^n}(1))} }
 \]
 The map $\rm\rmH^{3,1}_{\rmX \times \{0\}, M}(\rmX  \times {\mathbb A}^1, {\mathbb Z}/\ell^n) \ra \rmH^{3}_{\rmX \times \{0\}, et}(\rmX  \times {\mathbb A}^1, \mu_{\ell^n}(1))$
 identifies with the map 
 \[\rmH^{1,0}_{M}(\rmX \times  \{0\}, {\mathbb Z}/\ell^n) \ra \rmH^{1}_{et}(\rmX \times \{0\}, \mu_{\ell^n}(0))\]
 and $\rm\rmH^{1,0}_{M}(\rmX \times  \{0\}, {\mathbb Z}/\ell^n) \cong C\rmH^0(\rmX \times \{0\}, {Z}/\ell^n; -1) \cong 0$.
 Therefore this map is clearly injective. The map $\rm\rmH^{2,1}_{\rmX \times \{0\}, M}(\rmX  \times {\mathbb A}^1, {\mathbb Z}/\ell^n) \ra \rmH^{2}_{\rmX \times \{0\}, et}(\rmX  \times {\mathbb A}^1, \mu_{\ell^n}(1))$
 identifies with the map 
 \[\rmH^{0,0}_{M}(\rmX \times  \{0\}, {\mathbb Z}/\ell^n) \ra \rmH^{0}_{et}(\rmX \times \{0\}, \mu_{\ell^n}(0))\]
which is also an isomorphism. Now the required assertion follows from the following lemma.
\end{proof}
\begin{lemma}
 \label{diagm.chase.lemma}
 Consider the commutative diagram
 \[\xymatrix{{A'} \ar@<1ex>[r]^{f'} \ar@<1ex>[d]^{\alpha} & {B'} \ar@<1ex>[r]^{g'} \ar@<1ex>[d]^{\beta} & {C'} \ar@<1ex>[r]^{h'} \ar@<1ex>[d]^{\gamma} & {D'} \ar@<1ex>[r] \ar@<1ex>[d]^{\delta} &{E'} \ar@<1ex>[d]^{\eta}\\
             {A} \ar@<1ex>[r]^{f}  & {B} \ar@<1ex>[r]^{g}  & {C} \ar@<1ex>[r]^h  & {D} \ar@<1ex>[r]  &{E} }
   \]
with {\it exact} rows. Then the following hold:
\begin{enumerate}[\rm(i)]
 \item If $\alpha$ and $\beta$ are isomorphisms and $\delta$ is a monomorphism, then the map $\gamma$ is also a monomorphism.
 \item If $\alpha$ is an epimorphism and $\eta$ is a monomorphism, then 
 \[kernel(\beta) \ra kernel(\gamma) \ra kernel( \delta) \ra cokernel(\beta) \ra cokernel (\gamma) \ra cokernel( \delta)\]
 is exact. In particular, if  $\alpha$ is an epimorphism, $\eta$ is a monomorphism and both $\beta$ and $\delta$ are isomorphisms, then so is $\gamma$.
\end{enumerate}
\end{lemma}
\begin{proof} The proof of the first statement is a  straight-forward diagram-chase, making strong use of the fact $\alpha$ and
 $\beta$ are isomorphisms and $\delta$ is a monomorphism. Here is a an outline of a proof. Let $c' \eps C'$ be such that
 $\gamma(c')=0$. Then $\delta (h'(c')) =h(\gamma(c')) =0$. As $\delta$ is assumed to be a monomorphism, it follows $h'(c')=0$.
 By the exactness of the top row, there exists a $b' \eps B'$ so that $g'(b') = c'$. Now $ g(\beta(b')) = \gamma(g'(b')) = \gamma(c')=0$,
 so that there exists an $a \eps A$ so that $f(a) = \beta(b')$. But as both $\alpha $ and $\beta$ are isomorphism, there exists
 an $a' \eps A'$ so that $\alpha (a') = a$ and $f'(a') = b'$. But, then by the exactness of the top row, $c'=g'(b')=g'(f'(a'))=0$.
 Thus $\gamma$ must be a monomorphism, which proves the first statement.
 The second statement is a variant of the Snake Lemma: see, \cite[Snake Lemma 1.6]{Iver}.
\end{proof}

\vskip .2cm
We next recall the definition of the {\it torsion index} of connected linear algebraic groups from \cite[section 1]{Tot05}.
(Observe that since we assume the base field is separably closed, all linear algebraic groups we consider are split.)
Let $\rmH$ denote a fixed connected linear algebraic group with a chosen Borel subgroup $\rmB$ and a chosen maximal torus $\rmT \subseteq \rmB$.
Let $\rmN$ denote the dimension of $\rmH/\rmB$. For a linear algebraic group $\rmG$, we will let $\BG^{\rm gm}$ denote
$\BG^{\rm gm,m}$, for some $m>>0$.
\vskip .2cm
Next consider the diagram $\rmH/\rmB \ra  \BB^{\rm gm}  {\overset f \ra}  \BH^{\rm gm}$, where  $f$ denotes
the obvious map induced by the inclusion $\rmB \subseteq \rmH$. Observe that $\BB^{\rm gm}\simeq \BT^{\rm gm}$, where
 $\simeq$ denotes a weak-equivalence in the motivic homotopy category.
Then there exists a class $a \eps {\rm CH}^{\rm N}(\BB^{\rm gm}, {\rm Z}/\ell^n) (\cong \H^{2N, N}_M(\BB^{\rm gm}, {\mathbb Z}/\ell^n)$) so that $f_*(a) = t(\rmH) \in
 {\rm CH}^0(\BH^{\rm gm}, {\rm Z}/\ell^n) \cong {\rm Z}/\ell^n$. The class $t(\rmH)$ is the torsion index of $\rmH$.
\vskip .1cm
Next we consider the following diagram that commutes when the top and bottom rows denote maps going in the same direction:
\begin{equation}
\label{comm.square}
\xymatrix{{\H_{M}^{*, \bullet}(\BB^{\rm gm}, {\mathbb Z}/\ell^n)} \ar@<1ex>[r]^{f_*} \ar@<1ex>[d]^{cycl}_{\cong} & {\H_{M}^{*, \bullet}(\BH^{\rm gm}, {\mathbb Z}/\ell^n)} \ar@<1ex>[l]^{f^*} \ar@<1ex>[d]^{cycl}\\
            {\H_{et}^{*}(\BB^{\rm gm}, \mu_{\ell^n}(\bullet))} \ar@<1ex>[r]^{\bar f_*}  & {\H_{et}^{*}(\BH^{\rm gm}, \mu_{\ell^n}(\bullet))} \ar@<1ex>[l]^{\bar f^*}}
\end{equation}
In view of the fact that the cycle map is an isomorphism in the left column, and since the cycle map $\H_M^{0,0}(\BH^{\rm gm}, {\mathbb Z}/\ell^n) \ra \H_{et}^{0}(\BH^{\rm gm}, \mu_{\ell^n}(0)) \cong {\rm Z}/\ell^n$ is also
an isomorphism, one may define the torsion index similarly by starting with $\bar a =cycl(a)$.
 \vskip .1cm
 \begin{proposition} (See \cite[section 1]{Tot05}.)
  \label{torsion.index.prop}
The kernel and cokernel of the cycle map 
\[cycl:\H^{*, \bullet}_M(\BH^{\rm gm}, {\mathbb Z}/\ell^n) \ra H_{et}^*(\BH^{\rm gm}, \mu_{\ell^n}), \]
 as well as the kernel of the restriction map 
 \[ f^*: \H^{*, \bullet}_M(\BH^{\rm gm}, {\mathbb Z}/\ell^n) \ra H^{*, \bullet}_M(\BB^{\rm gm}, {\mathbb Z}/\ell^n)\]
 are killed by $t(\rmH)$.
  \end{proposition}
\begin{proof}
Define a map $\alpha: \CH^i(\BB^{\rm gm}, {\mathbb Z}/\ell^n) \cong \H^{2i,i}_M(\BB^{\rm gm}, {\mathbb Z}/\ell^n) \ra \CH^i(\BH^{\rm gm}, {\mathbb Z}/\ell^n ) \cong \H^{2i,i}_M(\BB^{\rm gm}, {\mathbb Z}/\ell^n)$
by $\alpha(x) = f_*(a.x)$. Then, $\alpha (f^*(x)) = f_*(a. f^*(x)) = f_*(a).x = t(\rmH).{\it x}$. As $\BT^{\rm gm}$ identifies with $\BB^{\rm gm}$,
the map $f^*$ identifies with the restriction homomorphism $res:\H^{*, \bullet}_M(\BH^{\rm gm}, {\mathbb Z}/\ell^n) \ra \H^{*, \bullet}_M(\BT^{\rm gm}, {\mathbb Z}/\ell^n)$,
thereby proving that its kernel is killed by the class $t(\rmH)$. In view of the fact that cycle map forming the left vertical map 
in ~\eqref{comm.square} is an isomorphism, it follows that the kernel of the cycle map
\begin{equation}
\label{cycl}
cycl: {\H_{M}^{*, \bullet}(\BH^{\rm gm}, {\mathbb Z}/\ell^n)} \ra {\H_{et}^{*}(\BH^{\rm gm}, \mu_{\ell^n}(\bullet))}
\end{equation}
is contained in the kernel of $f^*$, and hence is killed by the  class $t(\rmH)$.
\vskip .1cm
We next show the cokernel of the cycle map in ~\eqref{cycl} is also killed by the class $t(\rmH)$. Therefore, let $\bar x \in
{\H_{et}^{*}(\BH^{\rm gm}, \mu_{\ell^n}(\bullet))}$ denote a class. Then
\[t(\rmH). \bar {\it x}= \bar f_*(cycl(a). \bar f^*(\bar {\it x})) = \bar f_*(cycl({\it y})) = cycl(\bar f_*({\it y}))\]
for some class ${\it y} \in {\H_{M}^{*, \bullet}(\BB^{\rm gm}, {\mathbb Z}/\ell^n)}$. This shows the cokernel of the cycle map
in ~\eqref{cycl} is also killed by the class $t(\rmH)$, thereby completing the proof of the Proposition.
\end{proof}

\vskip .2cm
\section{Examples}
In this section, we discuss various examples making use of the techniques developed in the last section. We remind the reader that 
the base field $k$ is assumed to be separably closed, throughout.
\begin{example} The Brauer groups of the classifying spaces of connected linear algebraic groups.
\label{eg.1}
 \begin{theorem} Let $\rmH$ denote a connected linear algebraic group over $k$. Then $\Br(\BH)_{\ell^n} =0$ for any $\ell$ relatively prime
 to the torsion index $t(\rmH)$.
 \end{theorem}
 \begin{proof} First we invoke Theorem ~\ref{Br.grp.quot.stacks} to obtain the isomorphism $\Br(\BH)_{\ell^n} \cong \Br(\BH^{gm,m})$, $m>>0$.
 Next we invoke Theorem ~\ref{nonequiv.to.equiv}(iii), after identifying $\EG^{\rm gm,m}\times_{\rmG}(\rmG\times_{\rmH}\rmX)$ with
  ${\rm EH}^{\rm gm,m}\times_{\rmH}\rmX$, (with $\rmX = Spec \, {\it k}$) where $\rmG$ is a bigger group,  and containing $\rmH$ as a closed subgroup.
 \end{proof}
As examples, the torsion index for ${\rm GL}_n$ and ${\rm SL}_n$ are both $1$. The torsion index for ${\rm Sp}_{2n}$ and ${\rm Sp}_{2n+1}$ are powers of $2$,
 and the same holds for the orthogonal groups. The torsion indices for other classical groups are divisible only by the primes $2, 3, 5$.
 See \cite{Bor} for more details.
\end{example}
\begin{example} The Brauer group of the moduli stack of elliptic curves. Then we obtain the following result, which is
a restatement of Theorem ~\ref{mod.ell.curves}.
 \label{eg.1.1}
 \begin{theorem} 
 \label{ell.curves}
 Let ${\mathcal M}_{1,1}$ denote the moduli stack of elliptic curves over the base field $k$. Assume that
  $char(k) \ne 2,3$ and $\ell$ is a prime different from $char(k)$. Then $\Br({\mathcal M}_{1,1})_{\ell^n}=0$.
 \end{theorem}
\begin{proof}
 We observe from \cite[Proposition 28.6]{Ols} or \cite[Chapter IV section 4]{Hart77} that the stack ${\mathcal M}_{1,1} = [\rmY/{\mathbb G}_m]$, where $\rmY$ is the
 scheme $Spec \, k[g_2, g_3][1/\Delta] \subseteq {\mathbb A}^2_k$, where $\Delta = g_2^3-27g_3^2$. The action of ${\mathbb G}_m$
is given by $g_2 \mapsto u^4g_2, g_3 \mapsto u^6g_3$, $u \in {\mathbb G}_m$.
\vskip .1cm
Though $\rmY$ being open in ${\mathbb A}^2$ is rational, it is not projective and therefore it takes 
a bit of effort to show that $\Br(Y)_{\ell ^n}=0$ for any $\ell \ne char(k)$. Once that is done, Theorem ~\ref{Br.grp.quot.stacks} and 
Theorem ~\ref{nonequiv.to.equiv}(i) with $\rmG = {\mathbb G}_m$  proves the triviality of the $\ell^n$-torsion part of the 
Brauer group of the quotient stack $[\rmY/{\mathbb G}_m]$. 
\vskip .1cm
For the remainder of the proof,  we will let $x= g_2$, $y = g_3$ and $\tilde \Delta = Spec \, k[x, y]/(x^3-27y^2)$.
We begin with the commutative diagram of localization sequences:
 \[\xymatrix{\ar@<1ex>[r] &{\rmH^{2,1}_{M, \tilde \Delta}({\mathbb A}^2, {\mathbb Z}/\ell^n) } \ar@<1ex>[r] \ar@<1ex>[d] & {\rmH^{2, 1}_M({\mathbb A}^2,{\mathbb Z}/\ell^n)} \ar@<1ex>[r]  \ar@<1ex>[d]& {\rmH^{2, 1}_M({\mathbb A}^2-\tilde \Delta, {\mathbb Z}/\ell^n)} \ar@<1ex>[d] \ar@<1ex>[r]  &  {\rmH^{3,1}_{M, \tilde \Delta}({\mathbb A}^2, {\mathbb Z}/\ell^n)} \ar@<1ex>[d] \ar@<1ex>[r] &\\
  \ar@<1ex>[r] &{\rmH^{2}_{et, \tilde \Delta}({\mathbb A}^2, \mu_{\ell^n}(1))} \ar@<1ex>[r] & {\rmH^{2}_{et}({\mathbb A}^2, \mu_{\ell^n}(1))} \ar@<1ex>[r] & {\rmH^{2}_{et}({\mathbb A}^2- \tilde \Delta, \mu_{\ell^n}(1)) }\ar@<1ex>[r] & {\rmH^{3}_{et, \tilde \Delta}({\mathbb A}^2, \mu_{\ell^n}(1))}  \ar@<1ex>[r] &}
 \]
Using the identification $\rmH^{3,1}_{M, \tilde \Delta}({\mathbb A}^2, {\mathbb Z}/\ell^n) \cong CH^{1,-1}_{\tilde \Delta}({\mathbb A}^2, {\mathbb Z}/\ell^n)$, one sees that this term is trivial.
Moreover, one observes that 
\[\rmH^{2, 1}_M({\mathbb A}^2,{\mathbb Z}/\ell^n) \cong \rmH^{2}_{et}({\mathbb A}^2, \mu_{\ell^n}(1)) \cong 0.\]
In view of the commutative diagram above, therefore, now it suffices to show that $\rmH^{3}_{et, \tilde \Delta}({\mathbb A}^2, \mu_{\ell^n}(1))$ is trivial.
For this, we consider the long-exact sequence:
\begin{equation}
\label{long.exact.seq.1}
\xymatrix{ \ar@<1ex>[r] &{\rmH^{3}_{et, \{0\}}({\mathbb A}^2, \mu_{\ell^n}(1))} \ar@<1ex>[r] & {\rmH^{3}_{et, \tilde \Delta}({\mathbb A}^2, \mu_{\ell^n}(1))} \ar@<1ex>[r] & {\rmH^{3}_{et, \tilde \Delta-\{0\}}({\mathbb A}^2- \{0\}, \mu_{\ell^n}(1)) }} 
\end{equation}
\begin{equation}
\xymatrix{\ar@<1ex>[r]^(.2){\alpha}& {\rmH^{4}_{et, \{0\}}({\mathbb A}^2, \mu_{\ell^n}(1))}  \ar@<1ex>[r] & {\rmH^{4}_{et, \tilde \Delta}({\mathbb A}^2, \mu_{\ell^n}(1))} \ar@<1ex>[r] &}\notag
\end{equation}
Observe that curve corresponding to $\Delta$ has an isolated singularity at the origin, which can be resolved by taking
the normalization as follows. Observe that $\Delta$ corresponds to the plane curve with equation : $(x/3)^3= y^2$. Therefore, 
we substitute $(x/3)=t^2$ and $y=t^3$, so that $A= k[x,y]/((x^3-27y^2) \cong k[t^2, t^3]$ with function field $k(t)$. This is because
 $1/t= t^2/t^3 = (x/3)/y= x/(3y)$. But $k[t]$ is a unique factorization domain, so is already integrally closed. 
 Therefore, the integral closure of $A$ in $k(t)$ is $k[t]= k[3y/x]$, which corresponds to the affine line ${\mathbb A}^1$.
 This proves that the normalization of the curve $\tilde \Delta $ is the affine line ${\mathbb A}^1$ and the normalization maps 
 ${\mathbb A}^1-\{0\}$ isomorphically to the curve $\tilde \Delta -\{0\}$. Thus $\tilde \Delta -\{0\} \cong {\mathbb G}_m$ and 
 therefore,
 \[\rmH^{3}_{et, \tilde \Delta-\{0\}}({\mathbb A}^2- \{0\}, \mu_{\ell^n}(1)) = \rmH^{3}_{et, {\mathbb G}_m}({\mathbb A}^2- \{0\}, \mu_{\ell^n}(1)) \cong \rmH^{1}_{et}({\mathbb G}_m, \mu_{\ell^n}(0)) \cong {\mathbb Z}/\ell^n \mbox{ and }\]
 \[\rmH^{4}_{et, \{0\}}({\mathbb A}^2, \mu_{\ell^n}(1)) \cong {\mathbb Z}/\ell^n.\]
Observe that
 since the codimension of $\{0\}$ in ${\mathbb A}^2$ is $2$, $\rmH^{3}_{et, \{0\}}({\mathbb A}^2, \mu_{\ell^n}(1)) $ is trivial.
 Therefore, the long-exact sequence ~\eqref{long.exact.seq.1} will show that 
 \[\rmH^{3}_{et, \tilde \Delta}({\mathbb A}^2, \mu_{\ell^n}(1)) \cong 0,\]
{\it provided} that $\rmH^{4}_{et, \tilde \Delta}({\mathbb A}^2, \mu_{\ell^n}(1)) \cong 0$, so that the map denoted
 $\alpha$ in ~\eqref{long.exact.seq.1} is an isomorphism. For this, we consider the long-exact sequence:
 \begin{equation}
\label{long.exact.seq.2}
\xymatrix{ \ar@<1ex>[r] &{\rmH^{3}_{et}({\mathbb A}^2 - \tilde \Delta, \mu_{\ell^n}(1))} \ar@<1ex>[r] & {\rmH^{4}_{et, \tilde \Delta}({\mathbb A}^2, \mu_{\ell^n}(1))} \ar@<1ex>[r] & {\rmH^{4}_{et}({\mathbb A}^2, \mu_{\ell^n}(1)) } \ar@<1ex>[r] &} 
\end{equation}
Since $\rmH^{4}_{et}({\mathbb A}^2, \mu_{\ell^n}(1)) \cong 0$, it suffices to show $\rmH^{3}_{et}({\mathbb A}^2 - \tilde \Delta, \mu_{\ell^n}(1)) \cong 0$.
For this observe that 
\[\rmH^{3}_{et}({\mathbb A}^2 - \tilde \Delta, \mu_{\ell^n}(1)) \cong \rmH^{3}_{et}(({\mathbb A}^2 -\{0\})- (\tilde \Delta -\{0\}), \mu_{\ell^n}(1)) \cong \rmH^{3}_{et}({\mathbb A}^2 - {\mathbb A}^1, \mu_{\ell^n}(1)),\]
where the last isomorphism follows from the observation made earlier that the normalization of $\tilde \Delta$ is the affine line ${\mathbb A}^1$.
Finally, the long-exact sequence
\begin{equation}
\label{long.exact.seq.3}
\xymatrix{ \ar@<1ex>[r] &{\rmH^{3}_{et}({\mathbb A}^2, \mu_{\ell^n}(1))} \ar@<1ex>[r] & {\rmH^{3}_{et}({\mathbb A}^2 - {\mathbb A}^1, \mu_{\ell^n}(1))} \ar@<1ex>[r] & {\rmH^{4}_{et, {\mathbb A}^1}({\mathbb A}^2, \mu_{\ell^n}(1)) } \ar@<1ex>[r]&} 
\end{equation}
together with the isomorphism $\rmH^{4}_{et, {\mathbb A}^1}({\mathbb A}^2, \mu_{\ell^n}(1)) \cong \rmH^{2}_{et}({\mathbb A}^1, \mu_{\ell^n}(0)) \cong 0$
shows that $\rmH^{3}_{et}({\mathbb A}^2 - \tilde \Delta, \mu_{\ell^n}(1)) \cong \rmH^{3}_{et}({\mathbb A}^2 - {\mathbb A}^1, \mu_{\ell^n}(1)) \cong 0$.
This completes the proof of the theorem.
 \end{proof}
\begin{remark} Observe that our proof is much shorter and also works under the more general assumption that the base field
 is only separably closed than the proof in \cite{AM}. However, we require that the characteristic of the base field be different from both $2$ and $3$ so that it is
 possible to identify ${\mathcal M}_{1,1}$ with $[\rmY/{\mathbb G}_m]$. See also \cite{Shi}, who shows the Brauer group
 of ${\mathcal M}_{1,1}$ over an algebraically closed field of characteristic $2$, is ${\mathbb Z}/2$.
\end{remark}

\end{example}

\begin{example} Further examples of moduli stacks of principal bundles.
\normalfont
 \label{eg.2}
 Here we consider the following additional examples supplementing the discussion in Theorem ~\ref{main.thm.2}. We recall the
 somewhat conjectural formula for the (Voevodsky-)motive of the moduli stack ${\rm Bun}_{n,d}(\rmC)$ of rank $n$, degree $d$ vector bundles on a smooth projective curve $C$, with a $k$-rational point,
 as in \cite[Conjectures 1.3, 3.9]{HL}:
  \begin{equation}
 \label{Bun.1}
 \rmM({\rm Bun}_{n,d}(\rmC)) = M(Pic^d(\rmC)) \otimes \rmM(\rmB {\mathbb G}_m^{\rm gm}) \otimes \otimes_{i=1}^{n-1}Z(\rmC, {\rm Z}(i)[2i])
 \end{equation} 
 where $Z(\rmC, {\rm Z}(i)[2i]) = \oplus_{j=0}^{\infty} \rmM(\rmC^{(j)}) \otimes {\rm Z}(ij)[2ij]$ and 
 where $\otimes$ denotes the tensor product in the category of motives. (Recall this corresponds to the product of schemes.) 
 Here $\rmC^{(n)}$ denotes the
 $n$-fold symmetric power of the given curve. 
 When ${\mathcal L}$ is a fixed line bundle on the curve $\rmC$, ${\rm Bun}_{n,d}^{\mathcal L}$ will denote the moduli stack of rank $n$, degree $d$
 vector bundles on $\rmC$, with determinant isomorphic to ${\mathcal L}$. Then, it is shown in \cite{HL} that ~\eqref{Bun.1} specializes to
 \begin{equation}
 \label{Bun.L.1}
  \rmM({\rm Bun}_{n,d}^{\mathcal L}(\rmC)) = \rmM(\rmB {\mathbb G}_m^{\rm gm}) \otimes \otimes_{i=1}^{n-1}Z(\rmC, {\rm Z}(i)[2i]).
\end{equation}
\vskip .1cm
Similarly, the formula in ~\eqref{Bun.1} specializes to give the following formula for the moduli stack of principal ${\rm SL}_n$-bundles
 over $\rmC$:
 \begin{equation}
 \label{Bun.SLn}
  \rmM(Bun_{{\rm SL}_n}(\rmC)) =  \otimes_{i=1}^{n-1}Z(\rmC, {\rm Z}(i)[2i]).
\end{equation}
\begin{theorem}
 Assuming the above formulae for the motives of the above moduli stacks, and for a fixed prime $\ell \ne char(k)$, we obtain the following:
\begin{align}
  \Br(Bun_{n,d}(\rmC))_{\ell^n} &\cong \oplus_{j_1, \cdots, j_{n-1}=0}^{\infty}\Br(Pic^d(\rmC) \times \rmC^{(j_1)}\times \cdots \rmC^{(j_{n-1})})_{\ell^n},\\
  \Br(Bun_{n,d}^{\mathcal L}(\rmC))_{\ell ^n} &\cong \oplus_{j_1, \cdots j_{n-1}=0}^{\infty}\Br(\rmC^{(j_1)} \times \cdots \rmC^{(j_{n-1})})_{\ell^n},  \mbox{ and} \notag\\
  \Br(Bun_{{\rm SL}_n}(\rmC))_{\ell ^n} &\cong \oplus_{j_1, \cdots j_{n-1}=0}^{\infty}\Br(\rmC^{(j_1)} \times \cdots \rmC^{(j_{n-1})})_{\ell^n} \notag
 \end{align}

\end{theorem}
\begin{proof} The observation that $\rmB{\mathbb G}_m^{\rm gm} = \colimn {\mathbb P}^n$, and the fact that each ${\mathbb P}^n$ is a
 projective smooth linear scheme, shows that the Kunneth formula holds for the (usual) Chow groups of ${\rm B}{\mathbb G}_m^{\rm gm} \times \rmY$ for
 any smooth scheme $\rmY$: see \cite[Theorem 4.5, Corollary 4.6]{J01}. The corresponding statement also holds for \'etale cohomology. Moreover, the
 cycle map is an isomorphism for $\rmB{\mathbb G}_m^{\rm gm}$. Therefore, the $\rmB{\mathbb G}_m^{\rm gm}$ drops out of the Brauer groups. Similarly, 
 the factor ${\rm Z}(ij)[2ij]$ corresponds to the motive of a point shifted and Tate-twisted. Therefore, it also drops out of the Brauer groups
  resulting  in the
  formulae above.
\end{proof}
 \vskip .2cm
 \section{Brauer groups of GIT quotients: Proofs of Theorem ~\ref{main.thm.3} through ~\ref{Br.triv.2}}
 \vskip .2cm \noindent
{\bf Proof of Theorem ~\ref{main.thm.3}}
\vskip .1cm 
 First observe that the existence of the long-exact sequences ~\eqref{mot.exact.seq} and ~\eqref{et.exact.seq} is
 purely formal. Therefore, what needs to be shown is the surjectivity statements in ~\eqref{surj.ss}. The reason
 one often restricts to $\rmX$ projective (and smooth) is to ensure that various limits exits for actions of $1$-parameter subgroups.
 It is shown in \cite[Proposition 3.3, Theorem 3.4]{BJ12} that the required limits all exist in the case of quiver moduli.
 Therefore, with this observation the proof of the theorem discussed below, applies to both the cases, that is, where
 $\rmX$ is a smooth projective scheme and where $\rmX$ denotes the affine space of representations of a fixed quiver $\rmQ$ 
 with a fixed dimension vector ${\mathbf d}$.
 \vskip .1cm 
 We first observe the existence of the long exact sequences:
 \vskip .1cm 
 \begin{equation}
\label{mot.exact.seq.1}
\cdots \ra \H_{\rmG, \rmX- \rmU, \M}^{2, 1}(\rmX, Z/\ell^n) \ra \H_{\rmG, \M}^{2, 1}(X, Z/\ell^n) \ra 
 \rmH^{2, 1}_{\rmG, \M}( \rmU, Z/\ell^n) \ra  \H_{\rmG, \rmX- \rmU, \M}^{3, 1}(\rmX, Z/\ell^n) \ra \cdots,
 \end{equation} 
\begin{equation}
 \label{et.exact.seq.1}
\cdots \ra \H_{\rmG, \rmX- \rmU, et}^{2}(\rmX, \mu_{\ell^n}(1)) \ra \H_{\rmG, et}^{2}(X, \mu_{\ell^n}(1)) \ra 
 \rmH^{2}_{\rmG, et}( \rmU, \mu_{\ell^n}(1)) \ra \H_{\rmG, \rmX- \rmU, et}^{3}(\rmX, \mu_{\ell^n}(1)) \ra \cdots.
 \end{equation}
 Under the identification of motivic cohomology with the higher Chow groups, the first long exact sequence corresponds to the
 following long exact sequence:
 \begin{equation}
\label{mot.exact.seq.1.1}
\cdots \ra {\rm CH}^{1-c}_{\rmG}(\rmX- \rmU, 0, Z/\ell^n) \ra {\rm CH}^1_{\rmG}(X, 0, Z/\ell^n) \ra {\rm CH}^1_{\rmG}( \rmU, 0, Z/\ell^n) \ra  0,
 \end{equation} 
\vskip .1cm \noindent
where $c$ denotes the codimension of $\rmX - \rmU$ in $\rmX$, with
\[\H_{\rmG, \M}^{2, 1}(X, Z/\ell^n) \cong {\rm CH}^1_{\rmG}(X, 0, Z/\ell^n) \mbox { and } \rmH^{2,1}_{\rmG, \M}( \rmU, {\mathbb Z}/{\ell^n}) \cong {\rm CH}^1_{\rmG}( \rmU, 0, Z/\ell^n).\]
 Therefore, the map
\[\H_{\rmG, \M}^{2, 1}(X, Z/\ell^n) \ra \rmH^{2, 1}_{\rmG, \M}( \rmU, Z/\ell^n)\]
is always surjective irrespective of the codimension of $\rmX - \rmU$ in $\rmX$.
\vskip .2cm
Next we consider the proof of the theorem under the hypothesis in (a). Now it suffices to prove that
\[\H_{\rmG, \rmX- \rmU, et}^{3}(\rmX, \mu_{\ell^n}(1)) \cong 0\]
in case $codim_{\rmX}(\rmX - \rmU) \ge 2$. This is proved in \cite[Lemma 6.1]{J20}. We will nevertheless sketch  a proof for the 
convenience of the reader.
 Since the base field is assumed to be perfect, one may find an  open sub-variety $\rmX_0$ of $\rmX$ so that
$\rmY_0 = (\rmX-{\rm U})\cap \rmX_0$ is smooth and nonempty. Now one has a long-exact sequence in \'etale cohomology:
\begin{equation}
\label{long.ex.seq}
\cdots \ra \rmH^{j}_{ \rmY_1, et}(\rmX, \mu_{\ell^n}(1)) \ra \rmH^{j}_{\rmY,et}(\rmX, \mu_{\ell^n}(1)) \ra \rmH^{j}_{\rmY_0, et}(\rmX_0, \mu_{\ell^n}(1)) \ra \rmH^{j+1}_{\rmY_1, et}(\rmX, \mu_{\ell^n}(1)) \ra \cdots,
\end{equation}
where $\rmY= \rmX-{\rm U}$, $\rmY_1 = Y - Y_0$. We may also assume without loss of generality that $\rmY$ is  irreducible. Then
\begin{align}
\rmH^{j}_{\rmY_0, et}(\rmX_0, \mu_{\ell^n}(1)) &= \rmH^{j-2codim _{\rmX_0}(\rmY_0)}_{et}(\rmY_0, \mu_{\ell^n}(0)), \mbox { if } codim_{\rmX_0}(\rmY_0) =2\\
						       &= 0, \mbox {otherwise,}\notag
\end{align}
by Poincar\'e duality in \'etale cohomology.  Since we are working with mod-$\mu_{\ell^n}$-coefficients, with $\ell \ne char(k)$, 
the above groups
are trivial for $j-2codim _{\rmX_0}(\rmY_0)<0$, in particular for $j=3$.
\vskip .1cm
 Since $\rmY_1$ is of dimension strictly less than the dimension of
$\rmY$, an ascending induction on the dimension of $\rmY$ enables one to assume $\rmH^{j}_{\rmY_1, et}(\rmX, \mu_{\ell^n}(1)) =0$ for all $j <2codim _{\rmX}(\rmY_1)$. (One may start the induction when
$dim (\rmY)=0$, since in that case $\rmY$ is smooth.) Since
$codim _{\rmX}(\rmY) < codim _{\rmX}(\rmY_1)$, the long exact sequence in ~\eqref{long.ex.seq} now proves $ \rmH^j_{\rmY, et}(\rmX, \mu_{\ell^n}(1)) =0$ for all $j < 2codim _{\rmX}(\rmY)$.
This completes the proof of the theorem, under the hypothesis (a).
 \vskip .2cm
Next we will consider the proof of the theorem, under the hypothesis (b). In case $codim_{\rmX}(\rmX - \rmX^{ss})\ge 2$, we are in 
the situation already considered under the hypothesis (a) discussed in (a). Therefore,
if the highest dimensional strata in $\rmX - \rmX^{ss}$ have codimension 2 or higher, the conclusion follows.
Let $\{S_{\beta_o}|\beta_o \in \B_o\}$ denote the strata in $\rmX- \rmX^{ss}$ of the highest dimension, which we may assume are all of codimension $1$.
In case there are any remaining strata in $\rmX$ that are still unaccounted for, they are contained in $\rmX- (\rmX^{ss} \cup \sqcup_{\beta_o} S_{\beta_o})$,
so that the codimension of $\rmX- (\rmX^{ss} \cup \sqcup_{\beta_o \in \B_o} S_{\beta_o})$ in $\rmX$ is at least 2: this shows that the
restriction map 
\[\rmH^{2}_{\rmG, et}(\rmX, \mu_{\ell^n}(1)) \ra \rmH^{2}_{\rmG, et}(\rmX^{ss}\cup \sqcup_{\beta_o \in \B_o} S_{\beta_o}, \mu_{\ell^n}(1))\]
is surjective. In this case, let  $\rmX^o= \rmX^{ss} \cup \sqcup_{\beta_o \in \B_o} S_{\beta_o}$.
 Therefore, it suffices to show that the restriction map
\begin{equation}
\label{surj}
\rmH^{2}_{\rmG, et}(\rmX^{ss}\cup \sqcup_{\beta_o in \B_o} S_{\beta_o}, \mu_{\ell^n}(1)) \ra \rmH^{2}_{\rmG, et}(\rmX^{ss}, \mu_{\ell^n}(1))
\end{equation}
\vskip .1cm \noindent
is a surjection. 
In view of the long-exact sequence ~\eqref{et.exact.seq.1}, now it suffices to show the map 
\[\rmH^{3}_{\rmG, \rmX^o-\rmX^{ss}, et}(\rmX^o, \mu_{\ell^{\nu}}(1)) \ra \rmH^{3}_{\rmG, et}(\rmX^o, \mu_{\ell^{\nu}}(1))\]
is injective. In view of the isomorphism
\[\rmH^{3}_{\rmG, \rmX^o -\rmX^{ss}, et}(\rmX^o, \mu_{\ell^n}(1)) \cong \oplus_{\beta_o}\rmH^{1}_{\rmG, et}(S_{\beta_o}, \mu_{\ell^n}(0)),\]
now it suffices to show that the composite map
\[\rmH^{1}_{\rmG, et}(S_{\beta_o}, \mu_{\ell^n}(0)) \ra \rmH^{3}_{\rmG, et}(\rmX^o, \mu_{\ell^n}(1)) \ra \rmH^{3}_{\rmG, et}(S_{\beta_o}, \mu_{\ell^n}(1))\]
is injective, where the last map is the obvious restrictions to the stratum $\rmS_{\beta_o}$. Moreover,
the composite map is multiplication by the equivariant Euler class of the normal bundle to the imbedding of the 
stratum $S_{\beta_o}$ in $\rmX^o$.
\vskip .1cm
Now we recall from the introduction the following:
let $\rmY_{\beta_o}$ denote a locally closed subscheme of $\rmS_{\beta_o}$ so that it is stabilized
by a parabolic subgroup $\rmP_{\beta_o}$, with Levi factor $\rmL_{\beta_o}$. Moreover, then 
$\rmS_{\beta_o} \cong \rmG{\underset {\rmP_{\beta_o}} \times}Y_{\beta_o}^{ss}$, and there is a scheme $\rmZ_{\beta_o}$ with
an $\rmL_{\beta_o}$-action and an $\rmL_{\beta_o}$-equivariant Zariski-locally trivial surjection $\rmY_{\beta_o}^{ss} \ra \rmZ_{\beta_o}^{ss}$ whose 
fibers are affine spaces.  Moreover, $\rmZ _{\beta_o}$ is a smooth locally closed $\rmL_{\beta_o}$ -stable subscheme of $\rmX$, so that it is
a union of connected components of the fixed point scheme $\rmX^{\rmT_{\beta_o}}$, where $\rmT_{\beta_o}$ is a subtorus of $\rmG$ with
centralizer $\rmL_{\beta_o}$. 
(Here we need the assumption that the linearized action by $\rmG$ is {\it manageable} as in \cite[Theorem 4.7]{ADK}: recall that
this hypothesis always holds in characteristic $0$.)
\vskip .1cm
Then we also obtain the isomorphisms:
\begin{align}
 \rmH^{i}_{\rmG, et}(\rmS_{\beta_o}, \mu_{\ell^n}(j)) &\cong \rmH^{i}_{\rmG, et}(\rmG \times_{\rmP_{\beta_o}}\rmY_{\beta_o}^{ss}, \mu_{\ell^n}(j))\\
                                                                 &\cong \rmH^{i}_{\rmP_{\beta_o}, et}(\rmY^{ss}_{\beta_o}, \mu_{\ell^n}(j))\notag \\
                                                                 &\cong \rmH^{i}_{\rmL_{\beta_o}, et}(\rmZ^{ss}_{\beta_o}, \mu_{\ell^n}(j))\notag
\end{align}
By a criterion of Atiyah and Bott (see \cite[1.4]{AB83}),
it suffices to show that the equivariant Euler class of the normal bundle 
$\rmN_{\beta_o}$ to $\rmS_{\beta_o}$ in $\rmX$ is not a zero divisor in $\rmH^*_{\rmG, et}(\rmS_{\beta_o}, \mu_{\ell^n}(\bullet))$.
Under the above isomorphisms,  the equivariant Euler class of $\rmN_{\beta_o}$ identifies with that of
the restriction $\rmN_{\beta_o}\vert_{\rmZ_{\beta_o}^{ss}}$. But that restriction is
a quotient of the normal bundle $\rmN'_{\beta_o}$ to $\rmZ_{\beta_o}^{ss}$ in $\rmX$,
and the action of $\rmT_{\beta_o}$ on each fiber of $\rmN'_{\beta_o}$ has no non-zero fixed vector.
By Lemma ~\ref{lem:ab} below, it follows that the equivariant Euler class of $\rmN'_{\beta_o}$ is not 
a zero divisor in $\rmH^*_{\rmL_{\beta_o}}(\rmZ_{\beta_o}^{ss}, \mu_{\ell^n}(\bullet))$; thus, the same holds for the 
equivariant Euler class of $\rmN_{\beta_o}$. The condition on the order of the Weyl groups $\rmW_{\beta_o}$ is used
in showing that the equivariant cohomology with respect to $\rmL_{\beta_o}$ injects into the corresponding equivariant
cohomology with respect to a maximal torus in $\rmL_{\beta_o}$, so that the hypotheses of Lemma ~\ref{lem:ab} are satisfied.
\vskip .1cm
This completes the proof that the map in ~\eqref{surj} is a surjection and hence the proof of the theorem.
\qed
\vskip .1cm
\begin{lemma}
\label{lem:ab}
Let $\rmL$ be a linear algebraic group, $\rmZ$ an $\rmL$-variety, and $\rmN$ an $\rmL$-linearized vector
bundle on $\rmZ$. Assume that a subtorus $\rmT$ of $\rmL$ acts trivially on $\rmZ$ and fixes
no non-zero vectors in each fiber of $\rmN$. Then the equivariant Euler class of $\rmN$
is not a zero divisor in $\rmH^*_{\rmL, et}(\rmZ, \mu_{\ell^n}(\bullet))$ provided $|\rmW_{\rm L}|$ is prime to $\ell$,
where $\rmW_{\rm L}$ is the Weyl group associated to a maximal torus in $\rmL$.
\end{lemma}

\begin{proof}
We adapt the argument of \cite[13.4]{AB83}. Choose a maximal torus $\rmT_{\rmL}$ of $\rmL$
containing $\rmT$. Then the natural map $\rmH^*_{\rm L, et}(\rmZ, \mu_{\ell^n}(\bullet)) \to \rmH^*_{\rmT_{\rmL}}(\rmZ,\mu_{\ell^n}(\bullet))$
is injective as shown in Proposition ~\ref{transf.splitting} below. 
Thus, we may replace $\rmL$ with $\rmT_{\rm L}$, and assume that $\rmL$ is a torus.
Now $\rmL \cong \rmT \times \rmT'$ for some subtorus $\rmT'$ of $\rmL$. Therefore,
$\rmH^*_{\rmL, et}(\rmZ, \mu_{\ell^n}(\bullet)) \cong \rmH^*_{et}(\BT, \mu_{\ell^n}(\bullet)) \otimes \rmH^*_{\rmT', et}(\rmZ,  \mu_{\ell^n}(\bullet))$,
since $\rmT$ fixes $\rmZ$ point-wise. Moreover, $\rmN$ decomposes as a direct sum of 
$\rmL$-linearized vector bundles $\rmN_{\chi}$ on which $\rmT$ acts via a non-zero character 
$\chi$. Thus, we may further assume that $\rmN = \rmN_{\chi}$. Then the equivariant 
Euler class of $\rmN$ satisfies $c_d^L(N) =  \prod_{i=1}^d (\chi + \alpha_i)$,
where $d$ denotes the rank of $\rmN$, and $\alpha_i$ its $\rmT'$-equivariant
Chern roots. This is a non-zero divisor in 
$\rmH^*(\BT, \mu_{\ell^n}(\bullet)) \otimes \rmH^*_{\rmT', et}(\rmZ,\mu_{\ell^n}(\bullet ))$ since $\chi \neq 0$.
\end{proof}

\begin{proposition}
 \label{transf.splitting}
 Let $\rmL$ denote a linear algebraic group with $\rmT$ denoting a maximal torus in $\rmL$ and with $\rmW$ denoting
 the Weyl group of $\rmL$.  Let $\ell$ denote a fixed prime different from $char(k)$. Let $\rmY$ denote a smooth scheme
 over $k$ provided with an action by $\rmL$. Then, 
 if $|\rmW|$ is prime to $\ell$,
 the restriction map 
 \[\rmH^{*}_{\rmL, et}(\rmY, \mu_{\ell^n}(\bullet)) \ra \rmH^{*}_{\rmT, et}(\rmY, \mu_{\ell^n}(\bullet))\]
is a split monomorphisms. 
\end{proposition}
\begin{proof} Though this is discussed in  \cite{CJ20}, we will briefly recall the proof here.
 We obtain the identifications:
\[\rmH^{*}_{\rmL, et}(\rmY, \mu_{\ell^n}(\bullet))  \cong \rmH^{*}_{et}({\rm E}{\rmL }^{gm,m}\times_{ \rmL}{\rmY}, \mu_{\ell^n}(\bullet)) \mbox{ and }\]
\[\rmH^{*}_{\rmT, et}(\rmY, \mu_{\ell^n}(\bullet))  \cong \rmH^{*}_{et}({\rm E}{ \rmL}^{gm,m}\times_{ \rmL}( \rmL \times_{\rmT}\rmY), \mu_{\ell^n}(\bullet)) \cong \rmH^{*}_{et}({\rm E}{\rmL}^{gm,m}\times_{ \rmL}( \rmL/\rmT \times \rmY), \mu_{\ell^n}(\bullet)).\]
The last identification comes from the fact that since $\rmL$ acts on $\rmY$, there is a natural map $\rmL \times_{\rmT} \rmY \ra \rmL/\rmT \times \rmY$ that is
an isomorphism compatible with action by $\rmL$. Therefore, the transfer for $\rmL / \rmT$ now provides the required
splitting to the map induced by the projection of $\rmL/\rmT \ra Spec \, k$. This proves the proposition.
\end{proof}

\vskip .2cm \noindent
{\bf Proof of Theorem ~\ref{cor.2}.}
A key point is to observe the commutative square:
 \begin{equation}
\label{surj.ss.1}
\xymatrix{{\H_{\rmG, \M}^{2, 1}(X, Z/\ell^n)} \ar@<1ex>[r] \ar@<1ex>[d]^{cycl}& {\rmH^{2, 1}_{\rmG, \M}( \rmX^{ss}, Z/\ell^n)} \ar@<1ex>[d]^{cycl}\\
{\H_{\rmG, et}^{2}(X, \mu_{\ell^n}(1))} \ar@<1ex>[r]& {\rmH^{2}_{\rmG, et}( \rmX^{ss}, \mu_{\ell^n}(1))}.}
\end{equation}
Under the assumptions of the Theorem, both the horizontal maps are surjections as shown by Theorem ~\ref{main.thm.3}.
The left vertical map is an isomorphism as shown by Theorem ~\ref{nonequiv.to.equiv}(ii).
Now the commutativity of the above square shows the last vertical map is also a surjection, which proves the Theorem. 
\qed

\vskip .2cm \noindent
{\bf Proof of Theorem ~\ref{Br.triv.2}}.
\vskip .1cm
  The first observation is that, ${\rm Br}(\rmX//\rmG)_{\ell^n} = Br(\rmX^s/\rmG)_{\ell^n} \cong Br_{\rmG}(\rmX^s)_{\ell ^n}$ under the hypotheses
 of the corollary. Observe that the 
last isomorphism holds in view of the assumption that $\ell$ is prime to the orders of the stabilizer groups at all points
in $\rmX^{s}$. This then readily proves the theorem under the first hypothesis. 
\vskip .2cm
 Next we assume that the second hypothesis in the theorem holds. Now one needs to observe that for any finite degree approximation $\EG^{gm,m}$ to $\EG$, 
 ${\rm dim}(\EG^{gm,m}\times_{\rmG} (\rmX^{ss}- \rmX^s) = {\rm dim}(\BG^{gm,m}) + {\rm dim}(\rmX^{ss}- \rmX^s)$ while
 ${\rm dim}(\EG^{gm,m}\times_{\rmG} (\rmX^{ss}) = {\rm dim}(\BG^{gm,m}) + {\rm dim}(\rmX^{ss})$ so that 
 $codim_{\EG^{gm,m}\times_{\rmG} (\rmX^{ss})}(\EG^{gm,m}\times_{\rmG} (\rmX^{ss}-\rmX^s)) \ge 2$. Therefore, in 
 the long exact sequence
 \[\cdots \ra \rmH_{et, \rmG, \rmX^{ss}-\rmX^s}^2(\rmX^{ss}, \mu_{\ell^{\nu}}(1)) \ra \rmH_{et, \rmG}^2(\rmX^{ss}, \mu_{\ell^{\nu}}(1)) \ra \rmH_{et, \rmG}^2(\rmX^{s}, \mu_{\ell^{\nu}}(1)) \ra \rmH_{et, \rmG, \rmX^{ss}-\rmX^s}^3(\rmX^{ss}, \mu_{\ell^{\nu}}(1)) \ra \cdots\]
 the end terms are trivial. This provides the isomorphism:
  \begin{equation}
    \label{isom.ss.s}
 \rmH_{et, \rmG}^2(\rmX^{ss}, \mu_{\ell^{\nu}}(1)) {\overset {\cong} \ra} \rmH_{et, \rmG}^2(\rmX^{s}, \mu_{\ell^{\nu}}(1)).
 \end{equation} 
 Next, one considers the commutative diagram:
 \[\xymatrix{{0} \ar@<1ex>[r] & {\Pic(\EG^{gm,m}{\underset {\rmG} \times}{\rmX^{ss}})/\ell^n} \ar@<1ex>[d] \ar@<1ex>[r] & {\rmH^2_{et}(\EG^{gm,m}{\underset {\rmG} \times}{\rmX^{ss}}, \mu_{\ell^n}(1))} \ar@<1ex>[r] \ar@<1ex>[d] &  {\Br(\EG^{gm,m}{\underset {\rmG} \times}{\rmX^{ss}})_{\ell^n}} \ar@<1ex>[r]\ar@<1ex>[d] & {0} \\
             {0} \ar@<1ex>[r] & {\Pic(\EG^{gm,m}{\underset {\rmG} \times}{\rmX^{s}})/\ell^n} \ar@<1ex>[r] & {\rmH^2_{et}(\EG^{gm,m}{\underset {\rmG} \times}{\rmX^{s}}, \mu_{\ell^n}(1))} \ar@<1ex>[r] &  {\Br(\EG^{gm,m}{\underset {\rmG} \times}{\rmX^{s}})_{\ell^n} } \ar@<1ex>[r] & {0}.}
 \]
In view of the isomorphism in ~\eqref{isom.ss.s}, a five Lemma argument readily shows that the last vertical map is
also surjective. Therefore, it follows that the triviality of $\Br_{\rmG}(\rmX^{ss})_{\ell^n} = {\Br(\EG^{gm,m}{\underset {\rmG} \times}{\rmX^{ss}})_{\ell^n}}$ implies the triviality of 
$ {\Br(\EG^{gm,m}{\underset {\rmG} \times}{\rmX^{s}})_{\ell^n}} \cong  \Br_{\rmG}(\rmX^s)_{\ell^n} \cong \Br(\rmX//G)_{\ell^n} $, thereby completing the proof of (ii).  This completes the proof. \qed
 
\section{Examples}

\end{example}
\vskip .1cm \noindent
The next class of examples we consider will be the Brauer groups of various GIT-quotients. Throughout, we will assume the base field is algebraically closed.
\begin{example}
\label{prdct.grass}
	\normalfont
	The first example we consider, in this context is that of a product 
of Grassmannians:
$$
\rmX= \prod_{i=1}^m \Gr(r_i,n), \quad 
\cL = \boxtimes_{i=1}^m \cO_{\Gr(r_i,n)}(a_i)
$$
where $\Gr(r_i,n)$ denotes the Grassmannian of $r_i$-dimensional linear subspaces of 
projective $n$-space, and $\cO_{\Gr(r_i,n)}(a_i)$ denotes the $a_i$-th power of the 
line bundle associated with the Pl\"ucker embedding; here $\rmG =  \SL_{n+1}$ as in \cite[11.1]{Do03} and 
$r_1,\ldots,r_m < n$, $a_1,\ldots,a_m$ are positive integers, with $\rmG$  acting diagonally on $\rmX$. $\rmX^{ss} = \rmX^{s}$ for general values of 
$a_1,\ldots,a_m$, that is if $\Sigma_{i=1}^m a_i(r_i+1)$ and $n+1$ are relatively prime: see \cite[Section 11.1]{Do03}. The geometric quotient $\rmX/\!/\rmG$ 
is called the space of stable configurations; examples include moduli spaces
of $m$ ordered points in ${\mathbb P}^n$.
\vskip .1cm
One can readily see from \cite[Theorem 1.2]{J01} that the cycle map is an isomorphism for the product of Grassmannians,
so that Theorem ~\ref{nonequiv.to.equiv}(i) applies. One may conclude that $\Br_{\rmG}(\rmX^{ss})_{\ell^n} =0$ where $\ell$ is a sufficiently 
large prime satisfying the above hypothesis. Similarly, under the assumption that $\rmX^{ss} = \rmX^s$ and $\ell$ is prime
 to the orders of the stabilizers at points in $\rmX^s$, one also concludes that $\Br_{\rmG}(\rmX^s)_{\ell^n} \cong Br(X//\rmG)_{\ell^n} =0$.
\vskip .1cm
We continue to consider the various examples discussed in \cite[11.1]{Do03}, which will provide examples that
meet the hypotheses of Theorem ~\ref{Br.triv.2}. 
 \begin{enumerate}[\rm(i)]
  \item Take $n=2$, each $r_i=0$ and each $a_i =1$. In this case a point $(p_1, \cdots, p_m)$ is semi-stable if and only if
  no point is repeated more than $m/3$-times and no more than $2m/3$ points are on a line. In this case,
  $\rmX^{ss}=\rmX^s$ if $3$ does not divide $m$. (This is example 11.2 in \cite{Do03}.)
  \item Take $n=2$, $m=6$, each $r_i=0$ and each $a_i=1$. Then it is shown in the same example worked out in \cite{Do03} that
  $dim(\rmX^{ss}/\rmG) =4$, so that $dim(\rmX^{ss}) \ge 4+ dim(\rmG)$, and that $dim(\rmX^{ss}/\rmG - \rmX^s/\rmG)=1$,
  so that $dim(\rmX^{ss}-\rmX^s)\le 1+dim(\rmG)$. Thus in this case $codim_{\rmX^{ss}}(\rmX^{ss}- \rmX^s) \ge 2$ and
  $\rmX^s$ is non-empty.
  \item Take $n=3$, each $r_i=1$ and each $a_i=1$. Then we are considering sequences $(\ell_1, \cdots, \ell_m)$ of lines in 
  ${\mathbb P}^3$. The it is shown that $\rmX^s$ is empty if $m \le 4$. If $m=4$, it is shown that the dimension
  of $\rmX^{ss}$ is at least $2$. (In fact what is shown there is that the dimension of $\rmX^{ss}/\rmG \ge 2$, but this
  clearly implies that the dimension of $\rmX^{ss} $ is also at least $2$.)
  Therefore, in this case, $dim (\rmX^s)=0$ (as $\rmX^s$ is empty) and $dim(\rmX^{ss}- \rmX^s) = dim (\rmX^{ss}) \ge 2$.
 \end{enumerate}

 \end{example}
 \vskip .2cm

 \begin{example}
 \label{quiver.mod}
 	\normalfont
 The next example we consider is that of {\it Quiver moduli}. A {\it quiver} $\rmQ$ is a  finite directed graph, possibly with oriented cycles.
That is, $\rmQ$ is given by a finite set of vertices ${\rm I}$ 
(often also denoted $\rmQ_0$) and a finite set of arrows $\rmQ_1$. The arrows will be 
denoted by $\alpha:i\rightarrow j$. We will denote by 
$\bZ\bfI$ the free abelian group generated by $\rmI$; the basis consisting of elements 
of $\rmI$ will be denoted by $\bfI$. An element $\bfd \in \bZ\bfI$ will be written as 
$\bfd = \sum_{i\in I} d_i \, \bfi $. 

Let ${\rm Mod}(\bF \rmQ)$ denote the abelian category of finite-dimensional representations 
of $\rmQ$ over the finite field $\bF$ (or, equivalently, finite-dimensional representations 
of the path algebra $\bF \rmQ$). Its objects are thus given by tuples 
\begin{equation}
\label{quiver.rin}
\rmM=\big( (\rmM_i)_{i\in I},(\rmM_\alpha:\rmM_i\rightarrow \rmM_j)_{\alpha:i\rightarrow j} \big)
\end{equation} 

\medskip 

\noindent
of finite-dimensional $\bF$-vector spaces and $\bF$-linear maps between them. 

The {\it dimension vector} $\bfdim(\rmM) \in \bN \bfI$ is defined as 
$\bfdim(\rmM) = \sum_{i\in \rmI}\dim_{\bF }(\rmM_i) \, \bfi$. 
The {\it dimension} of $\rmM$ will be defined to be $\sum_{i \in I} \dim_{\bF }(\rmM_i)$, 
i.e. the sum of the dimensions of the $\bF$-vector spaces $\rmM_i$. This will be denoted 
$\dim(\rmM)$.

We denote by $\Hom_{\bF \rmQ}(\rmM, \rmN)$  the $\bF$-vector space of homomorphisms  
between two representations $\rmM, \rmN\in{\rm Mod}( \bF \rmQ )$.

We will fix a quiver $\rmQ$ and a dimension vector $\bfd = \sum_i d_i \, \bfi$,
and consider the affine space 
$$
\rmX=\rmR(\rmQ, \bfd) :=  \bigoplus_{\alpha:i\rightarrow j} \Hom_{\bF}(\bF^{d_i},\bF^{d_j}).
$$ 
Its points $\rmM=(\rmM_\alpha)_\alpha$ obviously parametrize representations of $\rmQ$ with 
dimension vector $\bfd$. (Strictly speaking only the $\bF$-rational points of $\rmX$ 
define such representations; in general, a point of $\rmX$ over a field extension
$k$ of $\bF$ will define only a representation of $\rmQ$ over $k$ with dimension
vector $\bfd$. We will however, ignore this issue for the most part.) 

The connected reductive algebraic group 
$$
\rmG(\rmQ, \bfd): = \prod_{i\in \rmI} \GL(d_i)
$$ 
acts on $\rmR(\rmQ, \bfd)$ via base change:
$$
\big( (g_i) \cdot (\rm M_\alpha) \big)_\alpha=(g_j \rmM_\alpha g_i^{-1})_{\alpha:i\rightarrow j}.
$$
By definition, the orbits $\rmG(\rmQ, \bfd)\cdot  M$ in 
$\rmR(\rmQ, \bfd)$ correspond bijectively to the isomorphism classes 
$[\rmM]$ of $\bF$-representations of $\rmQ$ of dimension vector $\bfd$. 
We will set for simplicity $\rmG := G(\rmQ,\bfd)$ and $\rmX := R(\rmQ,\bfd)$.
For any $\bar{\bF}$-rational point $\rmM$ of $\rmX$, the stabilizer
$\rmG_{\rmM} = \Aut_{\bar{\bF}Q}(M)$ is smooth and connected, since it is open in
the affine space $\End_{\bar{\bF}Q}(M)$. Also, note that the
subgroup of $\rmG$ consisting of tuples $(t \id_{d_i})_{i \in I}$, $t \in \bG_m$, 
is a central one-dimensional torus and acts trivially on $\rmX$; moreover, the 
quotient $PG(\rmQ,\bfd)$ by that subgroup acts faithfully.
So one may replace $\rmG$ henceforth by $PG(\rmQ, \bfd)$.

One may choose a linear function 
$\Theta:\bZ \bfI \rightarrow \bZ$ and associate to it a character 
$$
\chi_\Theta((g_i)_i) := \prod_{i\in I} \det(g_i)^{\Theta(\bfd)-\dim (\bfd)\cdot \Theta (\bfi)}
$$
of ${\rm PG}(\rmQ, \bfd)$. For convenience, we will call $\Theta$ itself a {\it character}. 
(This adjustment of $\Theta$ by a suitable multiple of the function 
$\dim : (d_i) \mapsto \sum_i d_i$ has the advantage that a fixed $\Theta$ can be used 
to formulate stability for arbitrary dimension vectors, and not only those with 
$\Theta(\bfd)=0$. However, this notation is a bit different from the one adopted in 
\cite{Kin94}.)

\medskip

Associated to each character $\Theta$, we define the {\it slope}  $\mu$. This is the function 
defined by $\mu(\bfd) = \frac{\Theta (\bfd)}{\dim (\bfd)}$.
With this framework, one may invoke the usual definitions of geometric invariant theory 
to define the semi-stable points and stable points. Observe that now
a point $x \in \rmR(\rmQ, \bfd)$ will be semi-stable (stable) precisely when there exists a 
$\rmG$-invariant global section of some positive power of the above line bundle that does not 
vanish at $x$ (when, in addition, the orbit of $x$ is closed in the semi-stable locus, 
and the stabilizer at $x$ is finite). Since all stabilizers are smooth and connected, 
the latter condition is equivalent to the stabilizer being trivial.

\medskip

The corresponding varieties of $\Theta$-semi-stable and stable points with respect to the line bundle 
${\rm L}_{\chi}$ will be denoted by
$$ 
\rmR(\rmQ, \bfd)^{ss} = \rmR(\rmQ, \bfd)^{\Theta-ss}=  \rmR(\rmQ, \bfd)^{\Theta-ss}
$$ 
and
$$ 
\rmR(\rmQ, \bfd)^{s} =  \rmR(\rmQ, \bfd)^{\Theta-s} =  \rmR(\rmQ, \bfd)^{\Theta-s}.
$$ 
These are open subvarieties of $\rmX$, possibly empty.
The corresponding quotient varieties will be denoted as follows:
$$
\rmM^{\Theta-s}(\rmQ, \bfd) = \rmR(\rmQ, \bfd)^{\Theta-s}/\rmG \mbox{ and } 
\rmM^{\Theta-ss}(\rmQ, \bfd) = \rmR(\rmQ, \bfd)^{\rm \Theta-ss}/\!/\rmG = \rmX/\!/\rmG.
$$ 
Observe that 
 the variety $\rmM^{\rm \Theta-s}(\rmQ, \bfd)$ parametrizes isomorphism classes 
of $\Theta$-stable representations of $\rmQ$ with dimension vector $\bfd$.
\vskip .2cm
\begin{proposition}
When every $\Theta$-semi-stable point is $\Theta$-stable, or if
\begin{equation}
\label{amply.stable}
codim_{\rmR(\rmQ, \bfd)^{\rm \Theta-ss}}(\rmR(\rmQ, \bfd)^{\rm \Theta-ss} - \rmR(\rmQ, \bfd)^{\Theta-s}) \ge 2,
\end{equation}
$\Br(\rmM^{\Theta-s}(\rmQ, \bfd))_{\ell^n}=0$ for $\ell$ sufficiently large.
In particular, the above conclusion holds if $gcd(\{d_i|i\}) =1$.
\end{proposition}
\begin{proof}
In this case, observe that $\rmX$ is the affine space $\rmR(\rmQ, \bfd)$. Therefore, Corollary ~\ref{Br.triv.2} applies to 
prove that $\Br(M^{\Theta-s}(\rmQ, \bfd))_{\ell^n}=0$ for $\ell$ sufficiently large, when every $\Theta$-semi-stable point is $\Theta$-stable
or if the hypothesis in ~\eqref{amply.stable} holds. Observe
that if $gcd(\{d_i|i\}) =1$, then every $\Theta$-semi-stable point is $\Theta$-stable.
\end{proof}
\vskip .1cm
\begin{remark} Here is a comparison of our result above with the results of \cite[Theorem 4.2]{RS}.
Note that in \cite[Theorem 4.2]{RS}, they consider the hypothesis:
\begin{equation}
\label{amply.stable.RS}
codim_{\rmR(\rmQ, \bfd)^{\rm \Theta}}(\rmR(\rmQ, \bfd)^{\rm \Theta} - \rmR(\rmQ, \bfd)^{\Theta-s}) \ge 2.
\end{equation}
Then they say that {\it the dimension vector $\bd$ is {\it an amply stable dimension vector}.}
This hypothesis is clearly stronger than the hypothesis ~\eqref{amply.stable} when $\rmR(\rmQ, \bfd)^{\rm \Theta-ss}$ is non-empty.
For then, $\rmR(\rmQ, \bfd)^{\rm \Theta-ss}$  being an open subscheme of $\rmR(\rmQ, \bfd)^{\rm \Theta}$ has the same dimension as $\rmR(\rmQ, \bfd)^{\rm \Theta}$, and
$\rmR(\rmQ, \bfd)^{\rm \Theta-ss} - \rmR(\rmQ, \bfd)^{\Theta-s} \subseteq \rmR(\rmQ, \bfd)^{\rm \Theta} - \rmR(\rmQ, \bfd)^{\Theta-s}$.
It is shown in \cite[Theorem 4.2]{RS}, that $\Br(\rmM^{\Theta-s}(\rmQ, \bfd))$ is cyclic of order $gcd(\{d_i|i\}$
under the  assumption that the hypothesis in ~\eqref{amply.stable.RS} holds. 
 In particular, if 
 $gcd(\{d_i|i\})=1$, then $\rmR(\rmQ, \bfd)^{\rm \Theta-ss} = \rmR(\rmQ, \bfd)^{\Theta-s}$, but still \cite[Theorem 4.2]{RS}
 seems to require that the hypothesis ~\eqref{amply.stable.RS} holds, in order to conclude that $\Br(\rmM^{\Theta-s}(\rmQ, \bfd))=0$.
 Our result above shows that 
$\Br(\rmM^{\Theta-s}(\rmQ, \bfd))_{\ell^n}=0$ if either $gcd(\{d_i|i\})=1$ or the hypothesis ~\eqref{amply.stable} holds, and 
{\it without} assuming the stronger hypothesis in ~\eqref{amply.stable.RS} holds, provided $\ell$ is sufficiently large.
\end{remark}
\end{example}
\vskip .2cm
Next we consider the assumptions made in Theorems ~\ref{cor.2} and ~\ref{Br.triv.2}, {\it on the codimension of the unstable locus} 
$\rmX -\rmX^{ss}$ in $\rmX$
{\it and whether} $\rmX^{ss}=\rmX^{s}$. The first observation is
that there are numerous classical examples in GIT, where $\rmX^{ss}=\rmX^s$, that is, every semi-stable point is stable: in addition to the examples considered
in Example ~\eqref{prdct.grass}, a well-known
 example is that of {\it binary forms of odd degree}. (See \cite[p. 110]{New}.)
\vskip .2cm
Next we consider the {\it codimension of the unstable locus}. It is important to point out that, in general, 
this depends on the choice of the $\rmG$-linearizing line-bundle and varies along with the variation of GIT-quotients. 
The following example illustrates this well.
\begin{example} {\it Unstable loci in flag varieties}: see \cite{ST}.
\normalfont
 \label{unst.loci.flags}
 Let $\rmG$ denote a complex reductive or semi-simple group with $\rmB$ a Borel subgroup. Let $\hat \rmG$ denote a 
 semi-simple subgroup of $\rmG$ acting on the flag variety $\rmX= \rmG/\rmB$. Let $\Lambda$ denote the character lattice of
 a maximal torus $\rmT \subseteq \rmB$. Then the ample line bundles on the flag variety $\rmX$ are given by the set of 
 strictly dominant weights denoted $\Lambda ^{++}$. Observe that ${\rm Pic}(\rmX) = \Lambda$. The $\hat \rmG$-ample cone
 $C^{\hat \rmG}(\rmX)_{\mathbb R}$ in ${\rm Pic}(\rmX)_{\mathbb R}$ is given by the line bundles that admit non-constant invariants in their section rings.
 \vskip .1cm
 One then obtains
an explicit description of the associated unstable locus for a line bundle ${\mathcal L}$ in  
$C^{\hat \rmG}(\rmX)_{\mathbb R}$ as well as a combinatorial formula for its co-dimension. It is shown 
 that the codimension is equal to $1$ on the regular boundary of the cone $C^{\hat \rmG}(\rmX)_{\mathbb R}$, and grows 
 towards the interior in steps by $1$, in a way that
the line bundles with unstable locus of codimension q form a convex
polyhedral cone.
\end{example}

\vskip .1cm


\end{document}